\documentclass[10pt]{article}
\usepackage{amsfonts}
\usepackage{amsfonts}

\usepackage{amsthm}
\usepackage{amsmath}
\usepackage{color}

\title{Ergodic pairs for degenerate pseudo Pucci's fully nonlinear operators} 

\author{ F. Demengel}

\date{}

\catcode`@=11 \@addtoreset{equation}{section} \catcode`@=12

\newtheorem{theorem}{Theorem}[section]
\newtheorem{proposition}{Proposition}
\newtheorem{remark}{Remark}

\newtheorem{lemma}{Lemma}

\def\grad{\nabla}

\begin{document}
\maketitle

\begin{abstract}
We study the ergodic problem for fully nonlinear elliptic  operators  $F( \nabla u, D^2 u)$ which may be degenerate when at least one of the components of the gradient  vanishes.  We extend  here the results in \cite{LL}, \cite{CDLP}, \cite{BDL2}, \cite{LP}.

\end{abstract}
\section{Introduction} 
This article deals with the existence of solutions to the ergodic problem associated to the "pseudo Pucci's" operators. 

 The history of the ergodic problem begins with the seminal paper of Lasry and Lions in 1989, \cite{LL} which considers the Laplacian case. More precisely $\Omega$ being an open  bounded $ \mathcal{C}^2$ domain in $ \mathbb R^N$, for $\beta \in ]1, 2]$ and $f$   being continuous in $\Omega$ and bounded,  $(u, c)$ is a solution of the ergodic problem  if 
  $$ \left\{ \begin{array}{lc}
  -\Delta u + | \nabla u |^\beta = f+ c& {\rm  in}  \ \Omega\\
   u = +\infty & {\rm on } \ \partial \Omega\end{array}\right.$$
   
   The results in \cite{LL} are extended to the  case of the $p$-Laplace operator by Leonori  and Porretta in \cite{LP}. In \cite{CDLP} the authors consider the case where the Laplacian  is replaced by $-{\rm div} ( A(x) \nabla u)$ where $A$   is positively definite and regular enough.   Recently  in  \cite {BDL2}, we considered  the case where the leading term is Fully Non Linear  elliptic, singular or degenerate, on the model of 
   $-| \nabla u |^\alpha F( D^2 u)$, where $\alpha >-1$ and $F$ is fully non linear elliptic and positively homogeneous of degree $1$. 
   
   \medskip
   
    In the present  paper, we will assume that 
   
   -There \ exist \ some \ constants \ $a< A$,  so\  that for \ any $M\in  S, N\in S, N \geq 0,$
    \begin{equation}\label{fnl}
  a tr(N) \leq F( M+ N)-F( M) \leq A tr N,
          \end{equation} 
      where $S$ is the space of symmetric matrices on $\mathbb R^N$.      We will also assume that 
      
      \begin{equation}\label{posh}
      F ( tM) =  t F( M), \ \forall t>0, \ M\in S.
      \end{equation} 
 We define  the operator ( with an obvious abuse of notations) 
   \begin{equation}
   \label{defF}
    F( \nabla u, D^2 u )=  F ( \Theta_\alpha (\nabla u ) D^2 u  \Theta_\alpha (\nabla u)), \  {\rm  where} \   \Theta_\alpha ( p) := {\rm Diag} ( |p_i|^{ \alpha\over 2})
   \end{equation} 
  $\alpha \geq 0$,  and $F$ satisfies \eqref{fnl}, \eqref{posh},   and we are interested  in the following  : 
  
     Let $\beta \in ]\alpha+1, \alpha+2]$,  find 
     $(u, c)$  which is a solution of the "ergodic problem " 
  $$ \left\{ \begin{array}{lc}
  -F( \nabla u, D^2 u ) + | \nabla u |^\beta = f+ c& {\rm in}  \ \Omega\\
   u = +\infty & {\rm on } \ \partial \Omega.\end{array}\right.$$
       Note that this equation presents a new type of degeneracy with respect to the equations in \cite{BDL2}, since the  leading term degenerates on every point where at least one derivative $\partial_i u$ is zero. 
     When $F( X) = tr X$, the operator is nothing else than the anisotropic $p$-Laplacian  for $p= \alpha+2$  ( also  called pseudo $p$-Laplacian or "orthotropic Laplacian"). Let us recall that  the    equation of the "anisotropic $p$-Laplacian",  is 
      \begin{equation}\label{plap}
      -\sum_i \partial_i ( | \partial_i u |^{p-2} \partial_i u) = f. 
      \end{equation}
     This equation can easily be solvable, for convenient $f$,  by standard methods in the calculus of variations. But  the regularity results are much more difficult to obtain. Lipschitz regularity is proved in the singular case in \cite{FFM} while the case $p>2$ is treated in \cite{BC},  \cite{BBJ} for a more degenerate equation  including the pseudo $p$Laplacian case. In \cite{D1}, \cite{BD2}  we considered  viscosity solutions for the  fully non linear extension of the pseudo$p$-Laplacian, say the  case where  in \eqref{plap}, the left hand side is replaced by  
       $-F ( \Theta_\alpha (\nabla u ) D^2 u  \Theta_\alpha (\nabla u))$, and $\alpha >0$. More general anisotropic  fully non linear degeneracy is treated in \cite{D2}.  In the variational case, one important   and recent  result  can be found in \cite{BB2}. 
       
        In  \cite{BD2} the Lipschitz interior regularity of the solutions is obtained as a corollary of the following classical estimate between $u$ sub-solution  and  $v$ super-solution of the  equation with some eventually different right hand side, say  :  If $\Omega = B(0,1)$, for all $r>0$, $r<1$, there exists $c_r$ so that for all $(x,y) \in B(0, r)^2$
        $$ u(x)-v(y) \leq \sup ( u-v) + c_r | x-y|.$$
        This Lipschitz estimate is extended  in the present paper to the equations presenting an Hamiltonian of the form $b(x) | \nabla u |^\beta$ with $\beta \in [\alpha+1, \alpha+2]$ when the sub-and super-solutions are bounded. This is done in Section 2. 
      This estimate does not permit to prove existence's results for the ergodic problem : This  existence is generally obtained by passing to the limit in an equation with boundary conditions  coming to $+\infty$, and then requires H\"older's  or Lipschitz estimates for  globally unbounded  solutions, ( though locally uniformly bounded).  However, as in \cite{LL}, \cite{CDLP}, \cite{BDL2}, the presence of an Hamiltonian  "superlinear" with a good sign,  permits to get an interior Lipschitz estimates for the \underline{solutions}, which does not require  that  the solution be bounded, but  that the zero order term   be so. This is done in Section \ref{secLipuni}. 
        
            One of the results of this paper is  : 
         
           \begin{theorem}\label{sympaetpas} Suppose that $f$ is  bounded and locally Lipschitz continuous  in $\Omega$, and that $F$ satisfies \eqref{fnl}, \eqref{posh} and  \eqref{defF}, and $\alpha \geq 0$. Consider the
Dirichlet problems 
             \begin{equation}\label{eq3}
             \left\{ \begin{array}{cl}
                     -F(\nabla u, D^2 u) + |\nabla u |^\beta   = f & \hbox{ in } \ \Omega \\
                     u=0 & \hbox{ on } \ \partial \Omega ,
                     \end{array}\right.
                     \end{equation}
and,  for $\lambda >0$, 
\begin{equation}\label{eq45}
\left\{ \begin{array}{cl}
                     - F(\nabla u , D^2 u) + |\nabla u |^\beta   + \lambda |u|^\alpha u  = f & \hbox{ in } \ \Omega \\
                     u=0 &  \hbox{ on } \ \partial \Omega .
                     \end{array}\right.
                     \end{equation}
The following alternative holds : 
\begin{enumerate}
\item Suppose that there exists a bounded sub-solution of \eqref{eq3}.
 Then the solution $u_\lambda$ of  \eqref{eq45}                      
satisfies: $(u_\lambda)$ is bounded and  uniformly converging up to a sequence $\lambda_n\to 0$ to  a  solution of (\ref{eq3}).
\item Suppose that there is  no solution for the Dirichlet problem \eqref{eq3}. Suppose in addition that $\alpha \geq 2$. 
 Then, $(u_\lambda)$  satisfies, up to a sequence $\lambda_n\to 0$ and locally uniformly in $\Omega$,
\begin{enumerate}
\item\label{a}  $u_\lambda\rightarrow -\infty$; 
\item\label{b} there exists a constant $c_\Omega\geq 0$ such that $\lambda |u_\lambda|^\alpha u_\lambda\to -c_\Omega;$
\item \label{c}  $c_\Omega $ is an ergodic constant  and  $v_\lambda=u_\lambda+|u_\lambda|_\infty$ converges  to  a solution of the ergodic problem    
\begin{equation}\label{ergodic}
\left\{ \begin{array}{cl}
                     - F(\nabla v, D^2 v) + |\nabla v |^\beta   = f +  c_\Omega  & \hbox{ in} \ \Omega \\
                     v=+\infty &  \hbox{ on} \ \partial \Omega 
                     \end{array}\right.
\end{equation}
whose minimum is zero.
\end{enumerate}
\end{enumerate}
\end{theorem}

     Note that, even when a  sub-soution to \eqref{eq3} exists, there  exists  an ergodic pair,    this will be proved in   Theorem \ref{exiergo},  section \ref{sexiergo}.

    We will also remark, without giving the details of the proofs,   that the ergodic constant can be characterized by an inf-formula analogous to the one which defines the principal eigenvalues for fully nonlinear operators. Following \cite{BPT}, \cite{P}, \cite{BDL2}, we define
          $$\mu^\star =  \inf \{ \mu\, :\,  \exists\,  \varphi \in  \mathcal {C} ( \overline{ \Omega}), - F(\nabla \varphi,  D^2 \varphi) + | \nabla \varphi |^\beta \leq f+ \mu \}\, .$$

  For the following theorem we introduce some new assumption : 
   \begin{equation}\label{C(x)}
   C(x) = \left((\gamma+1)F(\grad d, \grad d(x)\otimes\grad d(x))\right)^{\frac{1}{ \beta-\alpha-1}}\gamma^{-1} \ { \rm is} \ \mathcal {C}^2  \ {\rm in \ a \ neighborhood\ of } \ \partial \Omega.\end{equation}
    In particular (\ref{C(x)}) is satisfied when the boundary is $ \mathcal{ C}^3$ and $F$ is $ \mathcal {C}^2$. But it is automatically satisfied in the case where $F$ is one of the Pucci's operators. 
     We  will prove in  Theorem \ref{exiergo} that  under the assumption \eqref{C(x)}, any ergodic function  is equivalent near the boundary to $C(x) d^{-\gamma}$ and this allows to prove the uniqueness of the ergodic  constant in Therorem \ref{ABCDE} :

 \begin{theorem} \label{ABCDE}
Suppose that $f$ is  bounded and locally Lipschitz continuous  in $\Omega$, and that $F$ satisfies  \eqref{fnl}, \eqref{posh}  \eqref{defF},  and \eqref{C(x)}. Suppose that $\alpha \geq 2$. Let $c_\Omega$ be an ergodic constant for problem \eqref{ergodic}; then:
           \begin{enumerate}
            \item  $c_\Omega$ is unique; 
                       \item   $c_\Omega = \mu^\star$; 
 \item the map $\Omega\mapsto c_\Omega$ is  nondecreasing with respect to the  domain, and continuous;
\item  if  either $\alpha = 0$   or  $\alpha \neq 0$ and $\sup_\Omega f+ c_\Omega  <0$,  then  $\mu^\star$ is not achieved. 
Moreover,  if $ \Omega' \subset \subset \Omega$, then  $c_{\Omega' }< c_{ \Omega}$.  
\end{enumerate}
              \end{theorem}
              
              Theorems \ref{sympaetpas} and \ref{ABCDE} are obtained by means of several  intermediate  results, most of which are of independent 
interest. As we alaready mentioned it, 
 a first fundamental tool is an interior Lipschitz estimate for solutions of equation \eqref{eq45} that does not depend  on the $L^\infty$ norm 
of the solution but only on the norm of the zero order term. This is done in section \ref{secLipuni}. 

The uniqueness in Theorem \ref{ABCDE} is obtained, by using the results in Section \ref{compzero} : In this part,  we give a comparison theorem for sub- and super- solutions of equation \eqref{eq3}, in which zero order terms are lacking. 
The change of equation that allows to prove the 
comparison principle of Theorem \ref{2.4} is  standard, it has already been employed in \cite{LP} and \cite{BDL2}.
  
\medskip
Let us finally remark that  the question of uniqueness (up to constants) of the ergodic function is open.  We recall that the usual proof for linear operators, see \cite{LL,P},  relies on the strong comparison principle, which does not hold for degenerate operators. Let us also recall that for $p$-laplacian operators the uniqueness of the ergodic function is obtained in \cite{LP} for $p\geq 2$, in \cite{BDL2} for $p\leq 2$,   and in both cases  
 under the condition $\sup_\Omega f + c_\Omega < 0$. The same result is obtained in \cite{BDL2}. In all these degenerate cases, the $\mathcal{C}^1$ regularity is a crucial step, see  \cite{BDL3} for  the  equations  considered in \cite{BDL2}. In the present context of operators which degenerate  as soon as one derivative $\partial_i u$ is zero, the $ \mathcal {C}^1$ regularity  is known only in the case $N= 2$, \cite{BBr},  and more precisely only for equation\eqref{plap} and for $f = 0$, $p=\alpha+2$.
 The method  that the authors employ in \cite{BBr}  is very specific to the variational setting , since it relies essentially on very sharp Moser's iterations. 
  In a more recent paper,  \cite{LR}, the authors make precise the modulus of continuity of the gradient  and are able to generalize the $\mathcal{C}^1$ regularity to the more anisotropic equation 
    $$\partial_1 (  | \partial_1 u |^{ p_1-2} \partial_1 u) + \partial_2 (  | \partial_2 u |^{ p_2-2} \partial_1 u) =0$$
     when $p_1\geq 2, \ p_1\leq  p_2 < p_1+2$.

        \begin{remark} 
      The threshold value $\alpha = 2$ ( or $p = \alpha +2 = 4$)   appears in a lot of papers  treating of these anisotropic equations, let us cite in a non exhaustive manner \cite{BBJ}, \cite{UU}, The restriction $\alpha \leq 2$ or in some cases $\alpha \geq 2$  being  sometimes   relaxed  later.  
       We  are convinced that the results restricted by this condition here, hold  true without it, the fact that we cannot obtain them here is a lacking of the method employed. 
        \end{remark} 
  
    \begin{remark}
    In the equations considered  here, we will use the euclidian norm for the gradient term $| \nabla u |^\beta$, but  other norms should lead to  analogous results, with  obvious changes. 
    \end{remark}

\section{Existence results for the Dirichlet problem}

{\bf Notations} 
\begin{itemize}
\item In all the paper $| p|$ denote the euclidian norm of $p \in \mathbb R^N$. 
\item  $\Omega$ denotes a bounded $\mathcal{C}^2$ domain in $\mathbb R^N$.
 We use $d(x)$ to denote a  $\mathcal{C}^2$ positive function in $\Omega$ with coincides with the distance function from the boundary in a neighborhood of 
$\partial \Omega$.

\item For $\delta>0$, we  set $\Omega_\delta = \{ x\in \Omega\, :\, d(x)>\delta \}$.

\item We denote by $\mathcal{M}^+, \mathcal{M}^-$ the Pucci's operators with ellipticity constants $a, A$, namely, for all $M\in \mathcal {S}$,
$$
\begin{array}{c}
\mathcal{M}^+(M)= A\,  tr(M^+)-a\,  tr(M^-)\\[1ex]
\mathcal{M}^-(M)= a\,  tr(M^+)-A\,  tr(M^-)
\end{array}
$$
and we often use that, as a consequence of \eqref{fnl}, for all $M, N\in \mathcal {S}$ one has
$$
\mathcal{M}^-(N) \leq F(M+N)-F(M)\leq \mathcal{M}^+(N).
$$
\item We denote by $\overline{J}^{2,+} u( \bar x)$ ( resp. $\overline{J}^{2,-} v( \bar y)$) the upper closed semi-jets for a sub-solution $u$ at $\bar x$, ( resp.   the lower  semi-jets of a super-solution $v$ at $\bar y$), \cite{I}, \cite{usr}. 
\end{itemize}
 In some parts of the paper we will need the following properties of $F$  which are an easy consequence of the assumptions  \eqref{fnl},  and \eqref{defF} : 
 
  $(P_1)$ There exists $c$ so that for any $(p,q)\in \mathbb R^N$, and $M \in S$
  one has 
  $$ |F(p, M)-F( q, M)| \leq c ( |p|^{ \alpha\over 2} + |q|^{\alpha \over 2} )  \sum_1^N \left\vert |p_i|^{ \alpha\over 2}- |q_i |^{ \alpha \over 2}\right\vert    \left\vert M\right\vert .$$
   $ (P_2)$  There exists $c$ so that  for any $(p,q)\in \mathbb  R^N$, and $M \in S$,  diagonal 
   $$ |F(p, M)-F( q, M)| \leq c  \sum_1^N \left\vert  |p_i|^{ \alpha}- |q_i |^{ \alpha}\right\vert \  \left\vert M\right\vert.$$

The main result  of this section is the following : 

 \begin{theorem}\label{exidir}
  Suppose that $\alpha >0$, $\beta \leq \alpha+2$ and $F$ satisfies \eqref{fnl}, \eqref{posh}, \eqref{defF},   that $f$ and $b$ are  bounded. Suppose that $\lambda >0$. Then there exists a unique $u$  which satisfies 

 $$\left\{ \begin{array}{lc}
  -F(\nabla u, D^2 u) +b(x)  |\nabla u|^\beta + \lambda |u|^\alpha u  = f & {\rm in} \ \Omega\\
  u=0 & {\rm on} \ \partial \Omega
  \end{array}\right.$$
   Furthermore 
   $u $ is Lipschitz continuous, with some Lipschitz bound depending on $|u|_\infty, |f|_\infty, |b|_\infty$ in the case $\beta < \alpha+2$, and requires that $b$ be Lipschitz in the case $\beta = \alpha+2$.
   \end{theorem}

    It is classical that this existence's result is obtained by exhibiting  convenient sub-  and super-solutions, proving a Lipschitz estimate between them,   a comparison result,  and finally applying Perron's method adapted to the present context.   
     We will not give the details of all  the proofs, since the ideas here are a mixing of the arguments in \cite{BD1}, \cite {BDL1}. We enounce these results  :

 \begin{theorem}\label{theolip} Suppose that $F$ satisfies  \eqref{fnl}, \eqref{posh}, \eqref{defF}, that $\alpha>0$, $\alpha+1 \leq \beta \leq \alpha+2$,  that $b$ is continuous, and H\"older's continuous when $\beta = \alpha+2$. Suppose that  $u$ is a USC bounded by above  viscosity subsolution of 
      $$- F(\grad u,D^2 u)  + b(x)|\nabla u|^\beta \leq  g\quad \hbox{ in } B_1$$
and $v$ is a  LSC bounded by below   viscosity supersolution of 
$$- F(\nabla v, D^2 v)  + b(x)|\nabla v|^\beta \geq  f\quad \hbox{ in } B_1\, ,$$         
with $f$ and $g$ continuous and  bounded. 
Then,  for all $r<1$, there exists $c_r$ such that for all $(x,y)\in B_r^2$
        $$u(x)-v(y) \leq \sup_{B_1} (u-v)+ c_r |x-y|.$$
        \end{theorem}

    In order to prove Theorem \ref{theolip} we first need the following H\"older's estimate:
       
\begin{lemma}\label{lem1} Under the hypothesis of Theorem \ref{theolip},  for any    $\gamma \in (0,1)$,  and for all $r\in ]0,1[$,  there exists
$c_{r, \gamma}>0$ such that for  all $(x,y)\in B_r^2$
\begin{equation}\label{holdest}
u(x)-v(y) \leq \sup_{B_1} (u-v)+ c_{r, \gamma}|x-y|^\gamma .
        \end{equation} 
\end{lemma}        
\begin{proof}[Proof of Lemma \ref{lem1}]
We borrow ideas from \cite{IL}, \cite{BCI},  \cite{BDcocv}, \cite{BD2}. 
Fix $x_o \in B_r$, and define 
     $$\phi(x, y) = u(x) - v(y) -\sup_{B_1} ( u-v) -M|x-y|^\gamma -L (|x-x_o|^2+ |y-x_o|^2)$$
with $L= {16 (\sup u-\inf v)\over (1-r)^2}$  and $M= {4(\sup u-\inf v)\over \delta^\gamma}$,  
$\delta$ will be chosen later small enough depending only on the data and on universal constants.  
We want to prove that $\phi (x,y) \leq 0$ in $B_1$,  which will imply the result, 
taking first $x= x_o$ and making $x_o$ vary. 
   
We argue by contradiction and suppose that $\sup_{B_1}  \phi(x, y)>0$. 
By the previous assumptions on $M$ and $L$  the supremum is achieved on $(\bar x, \bar y)$  which  belongs to  
$B_{1+r\over 2} ^2$ and  is such that $0<|\bar x-\bar y| \leq \delta$. 
   
By Ishii's Lemma \cite{I}, \cite{usr},  there   exist  $X $  and $Y$ in $S$  such that 
$ (q^x, X ) \in \overline{J}^{2,+} u(\bar x), (q^y, -Y ) \in  \overline{J}^{2,-}v(\bar y)$
with 
           $$ q^x = \gamma M |\bar x-\bar y|^{\gamma-2} (\bar x-\bar y) +2 L (\bar x-x_o),$$
           $$q^y = \gamma M |\bar x-\bar y|^{\gamma-2} (\bar x-\bar y) -2L (\bar y-x_o), $$
        with 
            $$\left( \begin{array}{cc}
             X &0\\
             0& Y\end{array} \right) \leq 2 \left( \begin{array}{cc}
             B&-B\\
             -B&B\end{array}\right)$$
              and $B = D^2 (  |\cdot |^\gamma)$.  
Hence$$- F(q^x,X)  + b(\bar x)|q^x|^\beta \leq  g(\bar x),\quad           - F(q^y,-Y)  + b(\bar y)|q^y|^\beta \geq  f(\bar y).
$$
 Using the  computations  in \cite{BD2} one gets the existence of $c_1$ so that 
 $$ F(  q^x, X) \leq F(  q^y, -Y) -c_1 M^{1+ \alpha} |\bar x-\bar y |^{ ( \gamma-1) ( \alpha+1) -1}.$$
  So to conclude in the present case it is sufficient to obtain that for $\delta$ small, 
 $|b(x) |q_x|^\beta -b( y) |q_y|^\beta | $ is small with respect to $ M^{1+ \alpha} |\bar x-\bar y |^{ ( \gamma-1) ( \alpha+1) -1}$. This is obtained using 
 
  1) If $\beta < \alpha+2$
  \begin{eqnarray*}
   |b(x)-b(y) | | q_x |^\beta &\leq &2|b|_\infty M^\beta | \bar x-\bar y|^{( \gamma-1) \beta} \\
   &\leq&2|b|_\infty  | \bar x-\bar y |^{ 2+ \alpha-\beta} ( M^{1+ \alpha} |\bar x-\bar y |^{ ( \gamma-1) ( \alpha+1) -1})\\
   &<< &M^{1+ \alpha} |\bar x-\bar y |^{ ( \gamma-1) ( \alpha+1) -1}.
   \end{eqnarray*}
   
   2) If $\beta = \alpha+2$ we just use the continuity of $b$ 
  $$|b(x) -b( y)| | q_x|^\beta \leq o(1) M^{1+ \alpha} |\bar x-\bar y |^{ ( \gamma-1) ( \alpha+1) -1}.$$
On the other hand, by the mean value's Theorem, and for some universal constant $c$
  $$ |b( y)( | q_x |^\beta -|q_y|^\beta )| \leq  c |b|_\infty |q_x+q_y | M^{ \beta-1}|\bar x-\bar y |^{ ( \gamma-1) (\beta-1)}\leq 8cL |b|_\infty M^{ \beta-1}|\bar x-\bar y |^{ ( \gamma-1) (\beta-1)}$$
    which is also small with respect to $M^{1+ \alpha} |\bar x-\bar y |^{ ( \gamma-1) ( \alpha+1) -1}$.    We 
    can then conclude to a contradiction, since one has

               \begin{eqnarray*}
              -g(\bar x)  &\leq &  F(q^x,X) -b(\bar x) |q^x|^\beta  \\
              &\leq &    F(q^y,Y)   - c M^{1+ \alpha}  |\bar x-\bar y|^{\gamma-2+ ( \gamma-1) \alpha}- b ( \bar y) |q^y|^\beta  \\
              &\leq &-f(\bar y) - c \delta^{-\gamma(1+ \alpha)}  |\bar x-\bar y|^{\gamma(\alpha+1)-(2+\alpha)}\\
              &\leq & -f( \bar y) -c \delta^{-(2+\alpha)}.
              \end{eqnarray*}
This is a contradiction with the fact that $f$ and $g$ are bounded, as soon as $\delta$ is small enough.  
\end{proof}
\begin{proof} of Theorem \ref{theolip} 

For fixed $\tau\in (0,{\inf (1, \alpha) \over 2})$,  $\tau < \alpha+2-\beta$ when $\alpha+2-\beta>0$ and $\tau < \gamma_b$ where  $\gamma_b$ is some H\"older's exponent for $b$, when  $\beta = \alpha+2$.  Let  $s_o=(1+ \tau)^{\frac{1}{\tau}}$, and  define
for $s\in (0,s_o)$, $$
 \omega(s)  = s-{s^{1+\tau}\over 2(1+ \tau)},
$$ which we extend continuously after $s_o$ by a constant. 
Note that $\omega(s) $ is  $\mathcal{C}^2$ on $s>0$,  $s< s_o$,  satisfies 
 $\omega^\prime >{1\over 2} $, $\omega^{\prime \prime} <0$  on $]0,1[$, and 
 $s>\omega (s) \geq {s\over 2}$.

As before in the H\"older case, with $L = \frac{ 16( \sup u-\inf v)}{(1-r)^2}$ and $M = \frac{4( \sup u-\inf v)}{\delta}$,   we define  
$$\phi(x,y)= u(x) - v(y) -\sup_{B_1} (u-v) -M\omega(|x-y|) -L (|x-x_o|^2+ |y-x_o|^2).$$  
 Classically, as before, we suppose that there exists a maximum point $(\bar x, \bar y)$ such that $\phi(\bar x,\bar y)>0$, then by the assumptions on $M$, and $L$,  $\bar x, \bar y$ belong to $B( x_o, {1+r \over 2})$, hence they are interior points. 
This implies, using \eqref{holdest} in Lemma \ref{lem1} with $\gamma <1$ such that  $\frac{\gamma}{2} >  {\tau\over \inf (1,\alpha)}$ that,  for some constant  $c_r$,

\begin{equation}\label{hh} L | \bar y-x_o|^2, L|\bar x-x_o|^2 \leq c_r  |\bar x-\bar y |^{\gamma}.\end{equation}
and then  one has 
$| \bar y-x_o|, |\bar x-x_o | \leq \left({c_r  \over L}\right)^{1\over 2}  |\bar x-\bar y |^{\gamma\over 2}$.

Furthermore, 
 there   exist  $X $  and $Y $ in $S$  such that 
          $ (q^x, X ) \in \overline{J}^{2,+} u(\bar x)$, $(q^y, -Y ) \in \overline{J}^{2,-}v(\bar y)$ with
 $$ q^x= M \omega^\prime (|\bar x-\bar y|) {\bar x-\bar y\over |\bar x-\bar y|}  + L (\bar x-x_o), \ q^y= M \omega^\prime (|\bar x-\bar y|) {\bar x-\bar y\over |\bar x-\bar y|}  - L (\bar y-x_o).$$
  and 
     $$\left( \begin{array}{cc}
             X &0\\
             0& Y\end{array} \right) \leq 2 \left( \begin{array}{cc}
             B&-B\\
             -B&B\end{array}\right)$$
              and $B(x)  = D^2 (  \omega (| x|))$.  
  Following  the computations in \cite{BD2} ( for this we need among other things  \eqref{hh}),  one gets the existence of $c_1$ so that 
  $$ F( q^x, X)\leq F( q^y, -Y) \leq -c_1M^{1+ \alpha}  | \bar y-\bar x |^{\tau-1}$$
   So to conclude 
 we need to prove that 
  $|b(\bar x) |q_x|^\beta -b( \bar y) |q_y|^\beta | $ is small with respect to $ M^{1+ \alpha} | \bar y-\bar x |^{\tau-1}$. 
   This is obtained as in the H\"older's case by using  (the constant $c$  can vary from one line to another)

 1) If $\beta < \alpha +2$ 
   $$ |b(\bar x)-b( \bar y) | |q_x |^\beta \leq  2|b|_\infty M^\beta << M^{1+\alpha } | \bar x-\bar y |^{ \tau-1} $$
    by the assumption $\tau < 2+\alpha-\beta$. 
    
    2) If $ \beta = \alpha+2$
    $$ |b(\bar x)-b( \bar y) | |q_x |^\beta \leq c | \bar x-\bar y |^{\gamma_b} M^{2+\alpha} \leq c | \bar x-\bar y |^{ \gamma_b-\tau} ( M^{1+ \alpha} | \bar y-\bar x |^{\tau-1}).$$
     
   We finally use 
   $$ \left\vert b( \bar y) \right\vert  \left\vert  | q^y|^\beta -|q^x |^\beta \right\vert  \leq c|b|_\infty | q^y+ q^x| M^{\beta-1} \leq c M^{ \beta-1} . $$
    So the expected result holds by the choice of $M$ respectively to $\delta$. 
     Once more as in the proof of lemma \ref{lem1} one can conclude to
 a contradiction.

\end{proof}
It is clear that Theorem \ref{theolip}   can be extended to the case where $\Omega$ replaces $B(0, 1)$ and $\Omega^\prime \subset \subset \Omega$ replaces $B(0, r)$. Furthermore
adapting the method  in \cite{BDL1} that  we have the  following  Lipschitz estimate \underline{up to the boundary}  : 
 
  If $u$ is a sub-solution of 
  $$- F(  \nabla u ,D^2 u) + b(x) | \nabla u |^\beta \leq f, $$
   and $v$ is a super-solution of 
   $$- F( \nabla v , D^2 v) + b(x) | \nabla v |^\beta \geq g, $$
    and $u \leq 0$ , $v \geq 0$ on $\partial \Omega$, then there exists $c$ so that for any $(x,y) \in \overline{ \Omega}^2$
     $$u(x)-v(y) \leq \sup ( u-v) + c | x-y|. $$

 \begin{theorem} \label{thcompar}
 Suppose that $\Omega$ is a $\mathcal{ C}^2 $ bounded domain in $\mathbb  R^N$. Suppose that $\alpha>0$, $\alpha+1 \leq \beta\leq \alpha+2$, that  $F$ satisfies \eqref{fnl}, \eqref{posh} and \eqref{defF}.  Let $b$ be H\"older continuous. Let $\gamma$ be  a non decreasing  continuous function such that $\gamma(0)=0$.   Suppose that $u$ is a sub-solution of 
 $$-F( \nabla u, D^2 u) + b(x) | \nabla u |^\beta +\gamma( u) \leq g$$
  and 
 $$-F( \nabla v, D^2 v) + b(x) | \nabla v |^\beta + \gamma (v)  \leq f$$
  with $g \leq f$,  both of them being continuous and bounded. 
  Then if $ g< f$  in $\Omega$ or $\gamma$ is increasing, 
  if $u\leq v$ on $\partial \Omega$, $u \leq v$ in $\Omega$.
   \end{theorem}

    \begin{proof} of Theorem  \ref{thcompar}
    
     We use classically the doubling of variables. Suppose that $u> v$ somewhere, then consider 
     $$\psi_j(x, y ) = u(x)-v(y) -{j \over 2} |x-y|^2. $$
     Then for $j$ large enough the supremum of $\psi_j$ is  positive  and achieved on a  pair $(x_j, y_j) \in \Omega^2$, both of them converging to some maximum point $\bar x$ for $u-v$. Since $(x_j, y_j)$ converges to $( \bar x, \bar x)$, both of them  belong, for $j$ large enough, to some $\Omega^\prime \subset \subset \Omega$, independant on $j$.
 Furthermore, using  the Lipschitz estimate proved in Theorem \ref{theolip} : 
    $$  \sup ( u-v) \leq u(  x_j) -v(  y_j) - { j \over 2} | x_j- y_j|^2 \leq  \sup (u-v) + c_{ \Omega^\prime} |x_j-y_j| - { j \over 2} | x_j- y_j|^2,  $$
     from this one derives that 
     $j|x_j-y_j|$ is bounded. 
     
     Using Ishii's lemma , \cite{I},  \cite{usr},  there exist $X_j$ and $Y_j$ in $S$ such that 
             $(j(x_j-y_j), X_j) \in \overline{J}^{2,+} u(x_j)$, 
              $(j(x_j-y_j), -Y_j) \in \overline{J}^{2,-} v(y_j)$ and $X_j, Y_j$ satisfy 
               $$ -3j \left( \begin{array}{cc} 
                I &0\\0& I \end{array}\right) \leq \left( \begin{array}{cc} 
                X_j &0\\0& Y_j \end{array}\right) \leq 3j\left( \begin{array}{cc} 
                I &-I\\-I& I \end{array}\right). $$
We obtain,   denoting  the modulus of continuity of  $b$ by $\omega ( b , \delta)$ : 
             \begin{eqnarray*}
              g(x_j)- \gamma ( u(x_j))  &\geq & - F(j(x_j-y_j),X_j) +  b(x_j) | j(x_j-y_j)|^\beta\\
              &\geq & -F(j (x_j-y_j), -Y_j)  + b(y_j)   | j(x_j-y_j)|^\beta +c \omega (b, |x_j-y_j|) \\
                        &\geq &f(y_j) + o(1)  - \gamma (v(y_j)).\\
             \end{eqnarray*}
By passing to the limit,  one gets  on the point $\bar x$ limit of a subsequence of $x_j$ 
$$ g(\bar x) -\gamma ( u(\bar x))\geq f(\bar x)-\gamma(v(\bar x))$$
and in both cases we obtain a contradiction. 

\end{proof}

    \begin{proof} of Theorem \ref{exidir}
     We just give the hints to emphasize the difference with the operators and the results in \cite{BDL1}. We begin by  exhibit a sub- and a super-solution  which are zero on the boundary. 
     Suppose first that $\beta < \alpha+2$. Let us choose some constant $\kappa$  so that 
     $ \lambda \log(1+ \kappa) ^{1+ \alpha} > |f|_\infty$.  Let us  suppose $d <{ \kappa\over  C}$,  where $C$ will be chosen large enough depending on $|f|_\infty, |b|_\infty$ and on universal constants. 
      We can  assume that in $d < { \kappa\over C}$ the distance to the boundary is $ \mathcal{ C}^2$ and satisfies 
      $| \nabla d | = 1$. 
       Let us consider in $Cd < \kappa$ the function 
       $$\varphi(x) = \log ( 1+ Cd(x)).$$
        Then we have 
        \begin{eqnarray*}
         -F ( \nabla \varphi, D^2 \varphi) + b(x) | \nabla \varphi |^\beta &\geq& C^{2+\alpha} \left(a {\sum_1^N | \partial_i d |^{2+\alpha} \over( 1+ Cd)^{2+ \alpha}}- AC^{-1}  {|D^2 d|_\infty \sum_1^N |\partial_i d |^\alpha \over (1+ Cd)^{ \alpha+1}}\right) \\
         &-&C^\beta {|b|_\infty (\sum_1^N | \partial_i d |^2)^{ \beta \over 2}    \over (1+C d)^\beta} .
         \end{eqnarray*}
         Using the inequalities 
         
          $$\sum_i |\partial_i d |^{ \alpha+2} \leq (\sum_i |\partial_i d |^2)^{ \alpha+2 \over 2}, $$
           and 
           $$\sum_i |\partial_i d |^{ \beta} \leq (\sum_i |\partial_i d |^2)^{ \beta \over 2}, \ {\rm if} \ \beta>2, \ \sum_i |\partial_i d |^{ \beta}\leq (\sum \partial_i d^2 )^{ \beta \over 2} N^{1-{ \beta\over 2}}\ {\rm if\  not}, $$
            and analogous inequalities for $\sum | \partial_i d |^\alpha$,  one gets that for $Cd < \kappa$ there exist constants $\kappa_1, \kappa_2, \kappa_3,$  depending only on $a$, $A$ and on universal constants, so that 
          \begin{eqnarray*}
           -F ( \nabla \varphi, D^2 \varphi) &+& b(x) | \nabla \varphi |^\beta \\
           &\geq& \kappa_1 {C^{2+ \alpha} \over (1+ Cd)^{\alpha+2}}\left(1-C^{ \beta-\alpha-2} \kappa_2( 1+ Cd)^{ \beta-\alpha-2}- AC^{-1}  \kappa_3 (1+ Cd)^2\right) \\
           &\geq &{\kappa_1\over 2}   {C^{2+ \alpha} \over (1+\kappa )^{\alpha+2}}.
           \end{eqnarray*}
           as soon as $C$ is large enough,  more precisely  such that 
           $$C^{ \beta-\alpha-2} \kappa_2+AC^{-1}  \kappa_3 (1+ \kappa )^2<{1\over 2}$$
             and then assuming also $C$ so that 
           $  \kappa_1  {C^{2+ \alpha} \over (1+\kappa )^{\alpha+2}}>2 |f|_\infty$, we get that  $\varphi$ is a super-solution in $Cd < \kappa$. Extending it by $\log (1+ \kappa)$ in $Cd > \kappa$ and using the fact that the infimum of two super-solutions is a super-solution, we have the result. 
            To get a sub-solution take $-\varphi$ and adapt the constant. 
            
             Note that in the case $\beta = \alpha+2$ the previous  conclusion still holds if $|b|_\infty$ is small enough depending on universal constants. Note now that if $u$ is a supersolution of the equation 
             $$ -F(\nabla u, D^2 u) +b(x)  |\nabla u|^{ \alpha +2}  = \epsilon^{1+ \alpha}  f$$
              Then $u_\epsilon := { u \over \epsilon}$ satisfies 
              $$-F( \nabla u_\epsilon, D^2 u_\epsilon) + \epsilon b | \nabla u_\epsilon |^{ \alpha+2} = f$$
               and then a solution for the second problem gives one for the first one. 
               
                The existence and uniqueness  is then a direct consequence of the existence of these sub- and super-solutions and  of Perrron's method adpated to the context. We do not give the details.

                 \end{proof}

                \begin{remark}\label{nonhom}
                
                 In the sequel we will use a variant of this existence's result , that is to say, the boundary condition will  be $R$ in place of $0$. 
                 This can be done by taking for the super-solution $R+ \log ( 1+ Cd)$ in $Cd < \kappa$ extended by $R+ \log (1+ \kappa)$ in $Cd > \kappa$ and for the sub-solution  by taking for $k$ large enough $R-k \log (1+ Cd)$ in $Cd < \kappa$ extended by $R-k\log (1+ \kappa)$ in $Cd > \kappa$. 
                  \end{remark}

   \section{ Uniform Lipschitz estimates  when $\alpha \geq 2$ for unbounded solutions} \label{secLipuni}

    In this subsection we prove the following Lipschitz estimates : 
     \begin{proposition}\label{lipunif}  Let $F$ satisfy  \eqref{fnl},  \eqref{posh}  and \eqref{defF}, that   $\lambda\geq 0, $ $\alpha \geq 2$ and $\beta>\alpha+1$, let $u$ and $v$ be respectively a bounded  by above sub-solution and a bounded from below   super-solution of equation \eqref{eq45} in $B$, with $f$ Lipschitz continuous in $B$. Then, for any positive $p\geq \frac{(2+\alpha-\beta)^+}{\beta-\alpha-1}$, there exists a positive constant $M$, depending only on $p, \alpha, \beta, a, A, N, \|f-\lambda |u|^\alpha u\|_\infty$ and on the Lipschitz constant of $f$, such that, for all $x, y\in B$ one has
  $$
 u(x)-v(y)\leq \sup_B(u-v)^+ + M \frac{|x-y|}{(1-|y|)^{\frac{\beta}{\beta-\alpha-1}}} \left[ 1+\left( \frac{|x-y|}{(1-|x|)}\right)^p\right]
 $$
 \end{proposition}

         Due to the results in the previous subsection, in the case $\beta\leq 2+ \alpha $, the existence and uniqueness of $u_\lambda$ for equation \eqref{eq3} has been proved, and the Lipschitz bound  on $u_\lambda$  depends on the $L^\infty$ norm of $u_\lambda$ (more precisely on the oscillation of $u_\lambda$).  The  strength  of Proposition \ref{lipunif}  is that it provides bounds  on $u_\lambda$ independent on $\lambda$, as soon as $f-\lambda |u_\lambda|^\alpha u_\lambda $ is bounded. This will  allow  to  pass to the limit when $\lambda$ goes to zero in the next sections.

As in \cite{CDLP}, \cite{BDL2} it is sufficient to do the case  $\Omega = B(0,1)$. 
     
Let us define a "distance" function $d$ which  equals $1-|x|$ near the boundary and is extended as a smooth 
function   which has the properties 
      $$\left\{ \begin{array}{ccc}
      d(x) = 1-|x|& {\rm if} &\ |x| > {1\over 2}\\
        {1-|x|\over 2} \leq d(x) \leq 1-|x| & {\rm  for\  all} \ x\in \bar B& \ \\
         |Dd (x) | \leq 1 &-c_1 Id \leq  D^2 d (x) \leq 0 & \ {\rm for \ all} \ x\in \bar B
         \end{array} \right.$$
         for some constant $c_1>0$.

   Let us define as in \cite{CDLP}    $\xi = {|x-y| \over d(x)}$ and the function 
     $$\phi(x, y) = {k\over d(y)^\tau} |x-y| \left(L+ \xi^p \right)+ \sup ( u-v) $$
where $L$ and $k$ will be chosen large later, as well as $p$ and   $\tau$.     It is clear that if we prove that for such $k$ and $L$ one has  for all $(x,y) \in  B^2$ 
      
      $$ u(x)-v (y) \leq \phi(x, y), $$
       we are done. 
       
        So we suppose by contradiction that  $u(x)-v(y)-\phi(x,y)  >0$ somewhere,  then necessarily
          the supremum is  achieved on a pair $(x,y)$ with   $d(x) >0$,  $d(y) >0$ and $x \neq y$. Using Ishi's lemma, \cite{I}, \cite{usr}, one gets that on such a point, one has  for all $\epsilon>0$ the existence of  two symmetric matrices $X_\epsilon $ and $Y_\epsilon $,  such that 
        
          $$(\nabla_x \phi, X_\epsilon ) \in \overline{J}^{2,+} u( x), \ (-\nabla_y\phi, -Y_\epsilon ) \in \overline{J}^{2,-} v(y)$$ with 
           
         \begin{equation} \label{D2}-\left( {1\over \epsilon } + |D^2 \phi|\right) I_{2N} \leq \left(\begin{array}{cc} 
           X_\epsilon   & 0\\
           0& Y_\epsilon \end{array} \right) \leq D^2\phi + \epsilon  (D^2 \phi)^2.
           \end{equation} 
                     Since $u$ is a viscosity subsolution one has 
            $$ -F(\Theta_\alpha ( \nabla _x \phi ) X_\epsilon \Theta_\alpha ( \nabla _x \phi ) ) + |\nabla _x \phi |^{\beta} +\lambda|u|^\alpha  u(x) \leq f(x), $$
             while 
             $$ -F(-\Theta_\alpha ( \nabla _y\phi )Y_\epsilon  \Theta_\alpha ( \nabla _y \phi )) + |\nabla _y \phi |^{\beta} +\lambda |v|^\alpha v(y) \geq f(y).$$
Let   us multiply \eqref{D2} on the right  by 
              $$\left( \begin{array} {cc} 
             \Theta_\alpha( \nabla _x \phi ) & 0\\
              0&   \Theta_\alpha( \nabla _y\phi)  \end{array} \right)  \left( \begin{array} {cc} 
              \sqrt{1+t} \ I_N& 0\\
              0& I_N\end{array} \right)$$ where $t>0$,  and $I_N$ denotes the identity in $\mathbb R^N$, 
              and on the left, then we obtain that 
              
              \begin{eqnarray}\label{eqXY}
               &&\left( \begin{array} {cc} 
             \sqrt{1+t} \  \Theta_\alpha( \nabla _x \phi ) & 0\\
              0&   \Theta_\alpha( \nabla _y\phi)  \end{array} \right) \left(\begin{array}{cc} 
           X_\epsilon   & 0\\
           0& Y_\epsilon \end{array} \right)  \left( \begin{array} {cc} 
             \sqrt{1+t}\   \Theta_\alpha( \nabla _x \phi ) & 0\\
              0&   \Theta_\alpha( \nabla _y\phi)  \end{array} \right)  \nonumber\\
              &\leq& \left( \begin{array} {cc} 
              \sqrt{1+t} \  \Theta_\alpha( \nabla _x \phi ) & 0\\
              0&   \Theta_\alpha( \nabla _y\phi)  \end{array} \right) \left(D^2\phi \right) \left( \begin{array} {cc} 
            \sqrt{1+t}\  \Theta_\alpha( \nabla _x \phi ) & 0\\
              0&   \Theta_\alpha( \nabla _y\phi) 
               \end{array} \right) \nonumber\\
             & +&\left( \begin{array} {cc} 
            \sqrt{1+t} \  \Theta_\alpha( \nabla _x \phi ) & 0\\
              0&   \Theta_\alpha( \nabla _y\phi)  \end{array} \right) \left(\epsilon (D^2\phi)^2 \right)\left( \begin{array} {cc} 
             \sqrt{1+t}\  \Theta_\alpha( \nabla _x \phi ) & 0\\
              0&   \Theta_\alpha( \nabla _y\phi)  \end{array} \right)                \end{eqnarray}
               
               Note that by \eqref{posh},  and since $u$ and $v$ are respectively sub-and super-solutions,   one has 
                \begin{eqnarray*}
                 F(t  \Theta_\alpha( \nabla _x \phi )X_\epsilon  \Theta_\alpha( \nabla _x \phi ) ) &-& F((1+ t)  \Theta_\alpha( \nabla _x \phi )X_\epsilon \Theta_\alpha( \nabla _x \phi ) )\\
                 &+ &F(- \Theta_\alpha( \nabla _y\phi) Y_\epsilon \Theta_\alpha( \nabla _y\phi)  ) + \lambda (|u|^\alpha u (x)-|v|^\alpha v (y)) + |\nabla_x \phi|^\beta  \\
                 &-&|\nabla _y \phi|^\beta -f(x)+ f(y)\\
                 &\leq &0
                 \end{eqnarray*}
                 and then  using $u(x)-v(y) >0$
                 \begin{eqnarray*}
                  t  |\nabla_x \phi|^\beta & \leq& F( \nabla _x\phi, t X_\epsilon ) - t \lambda |u|^\alpha u (x)+ tf(x)  \nonumber\\
                  &\leq &F(\nabla _x\phi,  (1+t) X_\epsilon ) -F(\nabla _y \phi, -Y_\epsilon ) +  |\nabla_y \phi|^\beta   - |\nabla_x \phi|^\beta  +  t (f(x)-\lambda |u|^\alpha u (x))^+ \nonumber \\
                  &+&f(x )-f(y)\nonumber\\
                  &\leq& \mathcal{ M}^+ ( (1+ t)  \Theta_\alpha( \nabla _x \phi ) X_\epsilon  \Theta_\alpha( \nabla _x \phi )   + \Theta_\alpha( \nabla _y\phi )  Y_\epsilon   \Theta_\alpha( \nabla _y \phi ) ) \\
                  &+&  |\nabla_y \phi|^\beta -|\nabla_x \phi|^\beta +  t (f(x)-\lambda |u|^\alpha u (x))^+  + f(x)-f(y )\nonumber
                  \end{eqnarray*}
                                     Suppose that we get  an estimate of the form 
                   \begin{eqnarray}
                   \label{eqesti}\mathcal{ M}^+ ( (1+ t)  \Theta_\alpha( \nabla _x \phi ) X_\epsilon  \Theta_\alpha( \nabla _x \phi )  && + \Theta_\alpha( \nabla _y\phi )  Y_\epsilon   \Theta_\alpha( \nabla _y \phi )) \nonumber\\
                   &\leq& \psi (t, x, y,  D \phi, D^2 \phi)  + c \epsilon(1+t)\left(  |\Theta_\alpha ( \nabla_x\phi)|^2 \right. \nonumber \\
                & +&\left.  |\Theta_\alpha ( \nabla _y\phi)|^2\right)   |D^2 \phi|^2, 
                   \end{eqnarray}
                    for some function $\psi$,   then we  will  derive that 
                    \begin{eqnarray*}
                      t  |\nabla_x \phi|^\beta &\leq&   \psi (t, x, y,  D \phi, D^2 \phi) + c \epsilon(1+t) (|\Theta_\alpha ( \nabla_x\phi)|^2+ |\Theta_\alpha ( \nabla_y\phi)|^2 |D^2 \phi|^2\\
                      &+&  |\nabla_y \phi|^\beta - |\nabla_x \phi|^\beta 
                      +  t (f-\lambda |u|^\alpha u (x))^+  + f(x)-f(y ), 
                      \end{eqnarray*}
                                        and 
                    then letting $\epsilon$ go to $0$,  one gets 
                    $$  t  |\nabla_x \phi|^\beta \leq   \psi (t, x, y,  D \phi, D^2 \phi)+  |\nabla_y \phi|^\beta  - |\nabla_x \phi|^\beta +  t (f-\lambda |u|^\alpha u (x))^+  + f(x)-f(y ). $$    
                                                   Let us recall   some  useful estimates on  $\psi$    introduced   in   (\ref{eqesti}) : Using the computations  and the estimates in   \cite{CDLP}, one has  : 
$$\nabla_x \phi =  {k\over d(y)^\tau}\left( (L+ (1+ p) \xi^p) \eta -p \xi^{p+1} Dd (x)\right) , $$
                    where $\eta = {x-y\over |x-y|}$
                     and 
$$\nabla_y \phi = - {k\over d(y)^\tau}( L+ (1+ p) \xi^p) \eta+ \tau { k |x-y| \over  d(y)^{\tau+1}}\left( L + \xi^p\right) Dd(y).$$
                                   Note that one has 
                    $$|\nabla_x \phi|, |\nabla_y \phi | \leq c k {L+ \xi^{p+1}\over d(y)^{\tau+1}}  $$
                     and 
                      always like in  \cite{CDLP}   we can choose $L >1$  and large enough in order that 
                      $|\nabla_x \phi | \geq c k {(L+ \xi^p) (1+ \xi) \over d(y)^\tau}$. 
                                             We can sum up $D^2 \phi$ as follows 
                        \begin{eqnarray}\label{D2phi}
                        D^2 \phi &=&\gamma_1 \left( \begin{array} {cc} B& -B \\
                        -B& B  \end{array} \right) +\gamma_2 \left( \begin{array}{cc} 
                        T&-T\\
                        -T& T \end{array} \right)\nonumber\\
                        &+ &\gamma_3 \left( \begin{array} {cc} 
                        -(C+ ^t C ) &  C \\
                        ^t C& 0\end{array} \right)+\gamma_4 \left(\begin{array}{cc}
                        0&-^tD\\- D& (D+ ^t D) \end{array} \right)\nonumber \\
                        &+ &\left( \begin{array} {cc}
                         X_1& X_2 \\
                        X_3& X_4\end{array} \right)
                         \end{eqnarray}
                          with 
                           $B = I-\eta \otimes \eta, T = \eta\otimes \eta, C = \eta \otimes Dd(x), D = \eta \otimes Dd(y)$, and 
                                                             where 
                           $$\gamma_1 ={k\over d(y)^\tau}  { L+ ( 1+ p) \xi^p\over |x-y|},\ 
                            \gamma_2 = {k\over d(y)^\tau}  p(1+ p) { \xi^p \over |x-y|} ,\ 
                           \gamma_3 =  {k\over d(y)^\tau}  p (1+ p ) {\xi^p \over d(x)},$$
                          $$ \gamma_4 =  {k\over d(y)^\tau}  \tau {( L + (1+ p ) \xi^p )\over d(y)} $$
                            and $$X_1 ={k\over d(y)^\tau} \left(  {p( p +1) \xi^{p+1} \over d(x)} Dd (x) \otimes Dd (x) - p \xi^{p+1} D^2 d(x)\right), $$ 
                           $$  X_2 = {k\over d(y)^\tau}  {\tau p \xi^{p+1} \over d(y)}  Dd(x) \otimes Dd(y), \ 
                          X_3 =  {k\over d(y)^\tau}  { \tau p \xi^{p+1} \over d(y)}  Dd(y) \otimes Dd(x)$$
                             $$X_4 =  {k\over d(y)^\tau}  \left({ \tau (\tau+1) (L + \xi^p) |x-y| \over d(y)^2} Dd(y) \otimes Dd (y) -{\tau |x-y| \over d(y) } (L+ \xi^p) D^2 d(y) \right).$$
                                              Then    multiplying  ( \ref{D2phi})  by 
                                   $\left( \begin{array} {cc} \sqrt{1+t}& 0\\
                                   0& 1\end{array}\right) \left( \begin{array} {cc} 
             \Theta_\alpha( \nabla _x \phi ) & 0\\
              0&   \Theta_\alpha( \nabla _y\phi)  \end{array} \right) $ on the left and the right, one obtains                                                               

                        \begin{eqnarray*}
                         \psi(t,x, y, D\phi, D^2 \phi)&:=                         &\gamma_1 \left( \begin{array} {cc}(1+t)  \Theta_\alpha( \nabla _x \phi )B  \Theta_\alpha( \nabla _x \phi ) &-\sqrt{1+t} \  \Theta_\alpha( \nabla _x \phi )B \Theta_\alpha( \nabla _y \phi )\\
                             -\sqrt{ 1+t} \  \Theta_\alpha( \nabla _y \phi )B \Theta_\alpha( \nabla _x \phi ) &  \Theta_\alpha( \nabla _y \phi )B  \Theta_\alpha( \nabla _y \phi )\end{array} \right) \\
                             &+& \gamma_2 \left( \begin{array} {cc}(1+t)  \Theta_\alpha( \nabla _x \phi )T \Theta_\alpha( \nabla _x \phi ) &-\sqrt{1+t} \  \Theta_\alpha( \nabla _x \phi ) T \Theta_\alpha( \nabla _y\phi )\\
                             -\sqrt{ 1+t} \  \Theta_\alpha( \nabla _y \phi )T \Theta_\alpha( \nabla _x \phi ) &  \Theta_\alpha( \nabla _y \phi )T  \Theta_\alpha( \nabla _y \phi ) \end{array} \right) \\
                         &+& \gamma_3 \left(\begin{array}{cc}
                              -(1+t)  \Theta_\alpha( \nabla _x \phi )(C+ ^t C)  \Theta_\alpha( \nabla _x \phi )& \sqrt{1+t}\  \Theta_\alpha( \nabla _x \phi )^t C \Theta_\alpha( \nabla _y \phi )\\
                               \sqrt{1+t}\   \Theta_\alpha( \nabla _y \phi )C \Theta_\alpha( \nabla _x \phi ) & 0 \end{array} \right) \\
                              & +& \gamma_4 \left( \begin{array}{cc}
                              0& - \sqrt{1+t}\  \Theta_\alpha( \nabla _x \phi )D \Theta_\alpha( \nabla _y \phi )\\
                              - \sqrt{1+t}\  \Theta_\alpha( \nabla _y \phi )^t D \Theta_\alpha( \nabla _x \phi )& \Theta_\alpha( \nabla _y \phi ) (D+ ^t D )  \Theta_\alpha( \nabla _x \phi )\end{array} \right)\\
                             &+&  \left( \begin{array}{cc}  (1+t)  \Theta_\alpha( \nabla _x \phi )X_1 \Theta_\alpha( \nabla _x \phi )&  \sqrt{1+ t}  \Theta_\alpha( \nabla _x \phi )\ X_2 \Theta_\alpha( \nabla _y \phi ) \\
                             \sqrt{1+t} \   \Theta_\alpha( \nabla _y \phi )X_3 \Theta_\alpha( \nabla _x \phi ) &  \Theta_\alpha( \nabla _y \phi )X_4\Theta_\alpha( \nabla _y \phi ) \end{array} \right)
                              \end{eqnarray*}
           Multiplying the inequality ( \ref{eqXY}) by $( ^t v, ^t v)$ on the left and $\left(\begin{array}{c}v\\v\end{array}\right)$  on the right, where $v$ is any unit vector,  one gets  defining 
 $w _t= (\sqrt{1+ t} \  \Theta_\alpha ( \nabla_x\phi)  -   \Theta_\alpha ( \nabla _y \phi) )(   v)$ 
                                                           \begin{eqnarray}\label{eigenvalue}
                           ^t v ( (1+t)  \Theta_\alpha( \nabla _x \phi )X \Theta_\alpha( \nabla _x \phi ) &&+ \Theta_\alpha( \nabla _y \phi ) Y \Theta_\alpha( \nabla _y \phi )) v \nonumber\\
                           &\leq&\gamma_1 
                            ^t w_t   B w_t  +
               \gamma_2 ^t w_t  T  w_t+ c \gamma_3 t (|C|+ |^t C |)( | \Theta_\alpha ( \nabla_x \phi) |^2+ | \Theta_\alpha ( \nabla_y \phi) |^2) \nonumber\\
                           &+&   \gamma_4 t (|D|+ |^t D|)( | \Theta_\alpha (  \nabla_x  \phi) |^2+ | \Theta_\alpha ( \nabla_y \phi)|^2 ) \nonumber\\
                           &+&  \gamma_3( |C|+ |^t C| ) | \Theta_\alpha (  \nabla_x  \phi) - \Theta_\alpha ( \nabla_y \phi)  | \    | \Theta_\alpha (  \nabla_x  \phi)| + | \Theta_\alpha ( \nabla_y \phi)  | |v|^2\nonumber\\
                           &+ &\gamma_4 (|D|+ |^t D|)  | \Theta_\alpha (  \nabla_x  \phi) - \Theta_\alpha ( \nabla_y \phi)  |  ( | \Theta_\alpha (  \nabla_x  \phi) + \Theta_\alpha ( \nabla_y \phi))  | v|^2\nonumber\\
                           &+ &(\sqrt{1+t} ( |X_1| + |X_2|) + |X_3|) + |X_4|)    ( | \Theta_\alpha (  \nabla_x  \phi) +  \Theta_\alpha ( \nabla_y \phi)  | |v|^2.
                                                    \end{eqnarray}
                                        Note that 
                                                                                                    \begin{eqnarray*}
                                                      |w_t|^2 & \leq &2 ( (\sqrt{1+t}-1)^2  |\Theta_\alpha (  \nabla_x  \phi) |^2 + 2|\sum_i  |\partial_{i, _x} \phi |^{ \alpha\over 2}- |\partial_{i, _y}  \phi |^{ \alpha \over 2}|^2\\
                                                      &\leq  & 2 t^2 |\Theta_\alpha ( \nabla_x \phi) |^2 + 2|\sum_i  |\partial_{i, _x} \phi |^{ \alpha\over 2}-|\partial_{i, _y}  \phi |^{ \alpha \over 2}|^2. 
                                                       \end{eqnarray*}
                                Using ( \ref{eigenvalue}), every eigenvalue of  $(1+ t) X + Y$ satisfies 
                              $  \lambda_i (( 1+ t) \Theta_\alpha ( \nabla_x \phi)X \Theta_\alpha( \nabla _x \phi )+\Theta_\alpha( \nabla _y\phi ) Y\Theta_\alpha( \nabla _y \phi )) \leq c t^2 \Gamma_1 + t \Gamma_2 +( \gamma_1 + \gamma_2) \Gamma_4+  \Gamma_3$
                               for some  universal constant $c$, 
                             where we have denoted                                 $$\Gamma_1 =( \gamma_1+ \gamma_2)  |\Theta_\alpha ( \nabla_x \phi) |^2  \leq  ck^{1+ \alpha} {( L+ \xi^{p+1})^{1+\alpha } \over d(y)^{ \tau+ (\tau+1)\alpha} |x-y|}.$$
                                 
                                                         $$\Gamma_2:=  (\gamma_3 + \gamma_4+ |X_1 | + |X_2| + |X_3| ) (|\Theta_\alpha ( \nabla_x \phi) |^2  + |\Theta_\alpha ( D_y \phi) |^2 ) \leq 
                                                         ck^{1+ \alpha} {(L+  \xi^{p+2})( L+ \xi^{p+1}) ^\alpha \over d(y)^{(\tau+1)(1+ \alpha)}} .$$

                                                         $$\Gamma_3 := \sum_1^4 |X_i | \leq k^{1+ \alpha} |x-y| {( L + \xi^{p+2}) ( L + \xi^{p+1})^\alpha \over d(y)^{2+\tau+ (\tau+1) \alpha}}.$$
                         Finally $$\Gamma_4 :=  |\sum_i  |D_{i, _x} \phi |^{ \alpha\over 2}-| |D_{i, _y}  \phi |^{ \alpha \over 2}|^2.$$
       To majorize $\Gamma_4$ observe that         If $\xi \leq 1$, 
$${ |\xi|^{ p+1} |Dd | k  \over d(y)^\tau} \leq { \xi^p |x-y| \over d(x) d(y)^\tau} k \leq 2k|x-y|{ \xi^p
\over d(y)^{ \tau+1}}$$
while if $\xi \geq 1$,
 $${ \xi^{p +1} k \over d(y)^\tau} \leq  2k (1+ \xi^{p+1}) { |x-y| \over d(y)^{ \tau+1}}$$
  As a consequence 
  $|D_{ix} \phi+  D_{iy} \phi | \leq  c k ( L+ \xi^{p+1}){ | x-y|\over d(y)^{ \tau+1}}$. 
    
  Then  since  $\alpha \geq 2$ ( this is the only point where this restriction is  required),  one has  by the mean value's theorem 
  \begin{eqnarray*}
 \Gamma_4 = \sum_i |  |D_{ix} \phi|^{\alpha\over 2 } -|D_{iy }\phi|^{\alpha\over 2 }|^2&\leq & c\sum_i |D_{ix} \phi + D_{iy} \phi |^2 ( |D_{ix} \phi |^{{\alpha\over 2} -1} +  |D_{iy}  \phi |^{{ \alpha\over 2}  -1}) ^2\\
                                      &\leq &c { k^{\alpha} |x-y| ^2\over  d(y)^{(\tau+1)\alpha }}\left( L+ \xi^{p+1}\right)^{\alpha } \\
 \end{eqnarray*}

 and then 
$$\Gamma_4( \gamma_1 + \gamma_2) = | \sum_i(|D_{ix} \phi|^{\alpha \over 2} - |D_{iy} \phi|^{\alpha\over 2 })|^2 )  k {(L + \xi^{p+1})\over d(y)^\tau  |x-y|  }\leq c k^{\alpha+1} |x-y|{(L+ \xi^{p+1})^{\alpha +1} \over d(y)^{\tau ( \alpha+1)   + \alpha }}.$$
We now choose 
$ \tau > {2( \alpha+1 ) \over \beta-1-\alpha}  $, and  $p > { 2\alpha+3 \over \beta-( \alpha+1)}$, which imply   in particular that    by taking $k$ and $L$ large enough one has 
                                                               $\Gamma_2 <{|\nabla_x \phi |^{\beta}\over 2} $.    We have obtained that                                                                  
                                  
                                     $${ t\over 2}  |\nabla_x \phi|^{\beta} \leq t(f(x)-\lambda |u|^\alpha u (x)) + |D_y\phi|^{\beta} -|\nabla_x \phi |^{\beta} + (\Gamma_1t^2  + \Gamma_3+( \gamma_1+ \gamma_2) \Gamma_4).$$
                                                                                        Note now that we can choose $t$ optimal or equivalently 
                                  $t_o = { | \nabla_x  \phi|^\beta \over 4 \Gamma_1} $ and with this value of 
                                  $t_o$ one has 
                                   $$   | \nabla_x  \phi|^{2\beta}   \leq  16 \Gamma_1 t(f(x)-\lambda |u|^\alpha u (x)) +  \Gamma_1(|\nabla_y\phi|^{\beta} -| \nabla_x  \phi |^{\beta} )+  \Gamma_3\Gamma_1+ \Gamma_1 ( \gamma_1+ \gamma_2)  \Gamma_4).$$
                                                                         There remains to see that from this one derives a contradiction, indeed, 
                                     the left hand side is greater than 
                                      $ \left( { L + \xi^{p+1}\over d(y) ^\tau}\right)^{2 \beta} $ while 
                                      $$\Gamma_1 \Gamma_3 \leq c k^{2(1+ \alpha)} { ( L^{2(1+ \alpha)} + \xi^{ 2(p+2)( \alpha+1)} )\over d(y)^{ 2(\tau+1) ( \alpha +1)}}$$
                                       which is negligeable w.r.t. $| D _x \phi | ^{ 2 \beta}$ by the choice of $\tau$ and $p$.                                         Furthermore 
                                       \begin{eqnarray*}
                                       ( \gamma_1+ \gamma_2) \Gamma_4 \Gamma_1 &\leq& c  {k\over d(y)^\tau}  { L+ ( 1+ p) ^2\xi^p\over |x-y|}k^{1+ \alpha} {( L+ \xi^{p+1})^{1+\alpha } \over d(y)^{ \tau+ (\tau+1)\alpha} |x-y|}
                                       k^{ \alpha }  { |x-y|^2  ( L+ \xi^{p+1})^{\alpha}  \over ( d(y)^{( \tau+1) \alpha  }}\\
                                       &\leq & c k^{2(1+ \alpha)}{ ( L+ \xi^{p+1})^{ 2(1+ \alpha)} \over d(y)^{ 2\tau( \alpha+1)+ 2\alpha}}
                                       \end{eqnarray*}
                                       which is small w.r. t. $| \nabla _x \phi | ^{ 2 \beta}$ by the choice of $\tau$,  $p$, $k$ and $L$ .
                                        Finally 
                                        $$16 \Gamma_1 t |f-\lambda |u|^\alpha |_\infty   \leq 16 |  \nabla_x  \phi |^\beta  |f-\lambda |u|^\alpha |_\infty $$ which is small with respect to $|  \nabla_x  \phi|^{2 \beta}$ as soon as $L$ and $k$ are chosen large.                                        
                                      Furthermore 
                                      $$ | | \nabla_x  \phi |^{ \beta }-| \nabla_y \phi |^{ \beta }|\leq c |  \nabla_x  \phi+ \nabla_y \phi | (k {L+ \xi^{ p+1} \over d(y)^{\tau+1} })^{ {\beta }-1} \leq ck^{ \beta }  
                                      { |x-y|  ( L+ \xi^{p+1})^{\beta }  \over  d(y)^{ (1+ \tau){ \beta } }}$$
                                         and then 
                                                      
                                         $$| | \nabla_x  \phi |^{ \beta }-| D_y \phi |^{ \beta }|\Gamma_1 \leq 
                                         k^{ \beta+ 1+ \alpha} {( L + \xi^{p+1})^{\beta + 1+ \alpha} \over d(y)^{ \tau(  \alpha+1)+ \alpha + (\tau+1)  \beta }  }$$
                                          which is also small with respect to $|  \nabla_x  \phi |^{2 \beta}$ using the assumptions on $p$ and $\tau$.  
                                       We have obtained a contradiction.  Finally 
                                       $\phi(x,y) \leq 0$, 
                                       which implies  some Lipschitz estimate.                                            
                                          Arguing as in \cite{CDLP}, one can obtain an optimal behaviour of the  gradient of $u$ when $u$ is a solution , in the form 
              $$ | \nabla u _\lambda| \leq  cd^{-\gamma-1}. $$
              Furthermore,  these estimates can easily be extended to a $\mathcal{C}^2$  domain in place of a ball, using  the interior sphere property, and replacing of course 
               $1-|x|$ by $d(x, \partial \Omega)$.

 \section{Existence and behaviour near the boundary of ergodic function.}\label{sexiergo}
 In this section we prove the existence of solutions for  equation \eqref{eq3},  blowing up at the boundary, which will be used
in the proof of  existence of  ergodic pairs. In what follows we drop the assumption on the boundedness of the right hand side $f$, and we  consider continuous functions in $\Omega$, possibly unbounded as $d(x)\to 0$.  This  extension will be needed in particular to prove the  exact blow up estimate on the ergodic function in Theorem \ref{exiergo}.

\begin{theorem}\label{exilambda} Let  $\alpha \geq 0$,  $\beta  \in (\alpha+1,\alpha+2]$,  $\lambda >0$ and  let $F$ satisfy \eqref{fnl}, \eqref{posh}  and \eqref{defF}. Let further $f\in \mathcal{C}(\Omega)$  be  bounded from below  and such  that  
 \begin{equation}\label{f1}
\lim_{d(x)\to 0} f(x) d(x)^{\frac{\beta}{\beta-1-\alpha }} = 0\, .
\end{equation} 
 Then,   the infinite boundary value problem
                     \begin{equation}\label{eqlambdainfini}\left\{ \begin{array}{cl}
                     -| \nabla u |^\alpha F(D^2 u) + |\nabla u |^\beta   + \lambda |u|^\alpha u  = f &  \hbox{ in } \ \Omega \, ,\\
                     u=+\infty  & \hbox{ on } \ \partial \Omega\, ,
                     \end{array}\right.\end{equation}
admits solutions, and  any its solution $u$  satisfies, for all $x\in \Omega$,
\begin{equation}\label{bou}
\begin{array}{c}
\displaystyle \frac{c_0}{d(x)^\gamma}-\frac{D_1}{\lambda^{\frac{1}{\alpha+1}}} \leq u(x)\leq \frac{C_0}{d(x)^\gamma}+\frac{D_1}{\lambda^{\frac{1}{\alpha+1}}}
\qquad  \hbox{ if } \gamma>0\, ,\\[3ex]
\displaystyle c_0 |\log d(x)|-\frac{D_1}{\lambda^{\frac{1}{\alpha+1}}} \leq u(x)\leq  C_0 |\log d(x)|+\frac{D_1}{\lambda^{\frac{1}{\alpha+1}}}
\qquad  \hbox{ if } \gamma=0\, ,
\end{array}
\end{equation}
for positive constants $c_0, C_0$ and $D_1$ depending only on $\alpha, \beta, a, A, |d|_{\mathcal{C}^2(\Omega)}$ and on $|f|_\infty$.
\end{theorem}
 
   When $F$ satisfies furthermore   \eqref{C(x)} one has a better estimate : 

\begin{theorem}\label{precest} Let  $\beta  \in (\alpha+1,\alpha+2]$,  $\lambda >0$ and  let $F$ satisfy \eqref{defF}, \eqref{fnl}  and \eqref{C(x)}. Let further $f\in \mathcal{C}(\Omega)$  be  bounded from below  and such  that  
\begin{equation}\label{f}
\lim_{d(x)\to 0} f(x) d(x)^{\frac{\beta}{\beta-1-\alpha }-\gamma_0} = 0\, ,
\end{equation} 
for some $ \gamma_0\geq 0$.
 Then,  any solution $u$ of \eqref{eqlambdainfini} satisfies:  for any $\nu >0$ and for any $0\leq \gamma_1\leq \gamma_0$, with $\gamma_1< \inf (1, \alpha) $, and $\gamma_1<\gamma$ when $\gamma>0$, there exists $D=\frac{D_1}{\lambda^{1/(\alpha+1)}}$, with $D_1>0$ depending on $\nu, \gamma_1, \alpha, \beta, a, A, |d|_{\mathcal{C}^2(\Omega)}, |C(\cdot)|_{\mathcal{C}^2(\Omega)}$ and on $|f|_\infty$, such that, for all $x\in \Omega$,
\begin{equation}\label{bouest}
\begin{array}{c}
\displaystyle \frac{C(x)}{d(x)^\gamma}-\frac{\nu}{d(x)^{\gamma-\gamma_1}}-D\leq u(x)\leq \frac{C(x)}{d(x)^\gamma}+\frac{\nu}{d(x)^{\gamma-\gamma_1}}+D
\ \hbox{ if } \gamma>0\, ,\\[2ex]
|\log d(x)|\left( C(x) - \nu d(x)^{\gamma_1}\right)-D\leq u(x)\leq  |\log d(x)|\left( C(x) +\nu d(x)^{\gamma_1}\right)+D
\ \hbox{ if } \gamma=0\, .
\end{array}
\end{equation}
 Furthermore, the solution $u$ is unique.

\end{theorem}

   \begin{proof} of Theorem \ref{exilambda} : 
   
    In all the  proofs which follow, we will just detail the case $\gamma>0$ and leave the case $\gamma=0$ to the reader. 
Let $\delta$ be small enough in order that in $d < 2 \delta$ the distance is $ \mathcal {C}^2$. We define 

$$\varphi (x)= C_0 \, d(x)^{-\gamma} \, ,$$             then we have, by an easy computation, using the fact that $d$ is $ \mathcal {C}^2$ near the boundary and the properties of $F$ : 
$$
| F( \nabla \varphi , D^2 \varphi) - \gamma^{1+\alpha} ( \gamma+1) d^{-(\gamma+1) \alpha -\gamma-2} C_o^{1+\alpha} F( \nabla d, \nabla d \otimes \nabla d) | \leq c d^{-(\gamma+1) \alpha -\gamma-1}     , 
$$
and 
  $$ | \nabla \varphi |^\beta = C_o^\beta d^{-(\gamma+1)\beta} | \nabla d |^\beta $$
   so taking $C_o$ conveniently large  and using \eqref{f1}, 
   $$ -F( \nabla \varphi, D^2 \varphi)+ | \nabla \varphi |^\beta \geq f^+.$$
    In particular for any positive constant $D$ , $\varphi_1 : = \varphi+ D$ is also a supersolution of 
   $$  -F( \nabla \varphi_1, D^2 \varphi_1)+ | \nabla \varphi_1 |^\beta+ \lambda |  \varphi_1|^\alpha  \varphi_1 \geq f^+. $$
   We extend $\varphi_1$ inside $\Omega$ by taking 
   
    $$\overline{w} = \left\{ \begin{array}{cc}
    \varphi_1& {\rm in} \ d< \delta\\
     \frac{C_0}{\delta^\gamma} e^{\frac{1}{d(x)-2\delta}+\frac{1}{\delta}}+D\, ,& {\rm if} \ \delta < d \leq 2 \delta \\
     D & {\rm if} \ d > 2 \delta. 
     \end{array}\right.$$
      where, 
      since $\varphi_2$ is $\mathcal{C}^2$,  $K_1$   being  defined by 
      $$|F( \nabla \varphi_1, D^2 \varphi_1)|+ | \nabla \varphi_1 |^\beta \leq K_1, $$
  we have chosen  
      $D \geq \left( { |f|_\infty +K_1 \over \lambda } \right)^{1\over 1+\alpha}$. 
       Then $\overline{w}$ is a convenient  super-solution ( note that on $\delta < d \leq 2 \delta$ we use the fact that the infimum of two super-solutions is still a  super-solution). 
      
      Let us exhibit a convenient sub-solution : 
       For $s>0$, $c_0=\gamma^{-1}\left( \frac{(\gamma+1) a}{2}\right)^{\frac{1}{\beta-\alpha-1}}$ and $x\in \Omega\setminus \Omega_{2\delta}$, let us consider the function
 $\varphi ^s (x)= c_0 (d(x)+ s)^{-\gamma}\, .$
One has 
$$
-  F (\nabla \varphi^s,  D^2 \varphi^s)+ |\nabla \varphi^s|^\beta +\lambda (\varphi^s)^{1+\alpha} 
 \leq f(x)\, ,
$$
for $\delta$ and $s$ sufficiently small, since $f$ is bounded from below.

Moreover, for $D\geq c_0 \delta^{-\gamma}+\left(\frac{|f^-|_\infty}{\lambda}\right)^{\frac{1}{\alpha+1}}$,  the constant function 
$c_0 (\delta+s)^{-\gamma}-D$ is also a sub solution in $\Omega$.

Therefore, the function
             $$ \underline{w}^s(x) = \left\{ \begin{array} {ll}
     \varphi^s(x)-D & {\rm in } \   \Omega\setminus \Omega_\delta\\[2ex]
     c_0 ( \delta+s)^{-\gamma} -D  & { \rm in } \ \Omega_\delta
     \end{array} \right.
     $$
   is a convenient  sub-solution.

  Using Remark \ref{nonhom} after the existence Theorem  \ref{exidir},     let $u_R$ which satisfies 
  \begin{equation}\label{uR}\left\{ \begin{array}{lc}
  - F(  D u_R, D^2 u_R) + | \nabla u_R |^\beta + \lambda |u_R|^\alpha u_R = f _R& {\rm in} \ \Omega\\
   u_R=R & {\rm on} \ \partial \Omega, 
   \end{array}
   \right.
   \end{equation} 
   where $f_R = \inf ( f, R)$. 
    Let us  observe,   using the comparison principle    in Theorem \ref{thcompar},   that 
     $\underline{w}^s \leq u_R \leq \overline{w}, $   that $(u_R)_R$ is non decreasing,   and since it is trapped between  $\underline{w}^s $ and $\overline{w}^s $, $(u_R)_R$  is locally uniformly bounded, hence locally uniformly Lipschitz.  By classical results for uniformly  Lipschitz  viscosity solutions,  it converges  to a solution of ( \ref{eqlambdainfini}) inside $\Omega$. The boundary behaviour follows by letting $s$ go to zero. 
      \medskip
      
       Note that here we did not use the uniform Lipschitz estimates, so  neither $\alpha \geq 2$, nor  "$f$ Lipschitz continuous"   is  needed. 
       However, we do not have the precise estimate at the boundary, and then we cannot ensure  the uniqueness.

\end{proof}   

 \begin{proof} of Theorem \ref{precest}
 
            We     introduce for $\delta >0$ small 
                       $$ \varphi_1 = \left(  \left({ F( \nabla d, \nabla d \otimes \nabla d) ( \gamma+1)\over \gamma^{\beta-\alpha-1}} \right)^{1\over \beta-\alpha-1}+\nu d^{\gamma_1} \right) d^{-\gamma} 
                       + D$$
                         and 
                        $$ \underline{ w}_{\epsilon, \delta}= \left( \left({ F( \nabla d, \nabla d \otimes \nabla d)  ( \gamma+1)\over \gamma^{\beta-\alpha-1} }\right)^{1\over \beta-\alpha-1} -\nu d^{\gamma_1}  \right) (d+ \delta) ^{-\gamma} - D$$
                        where $D  $  is  some constant to be chosen later . 
                        Recall that $C(x)=\left({ F( \nabla d, \nabla d \otimes \nabla d) ( \gamma+1)\over \gamma^{\beta-\alpha -1}} \right)^{1\over \beta-\alpha-1}$.  
                        
We  prove that $ \varphi_1 $ is a supersolution in $\Omega_{2\delta} = \{ x \in \Omega, d(x, \partial \Omega \leq 2\delta\}$ for $\delta $ small enough. We denote   $ \varphi (x) =  C(x) d^{ -\gamma}+ \nu d^{-\gamma+\gamma_1}$.   Then $D \varphi = -\gamma  C(x) d^{-\gamma-1} \nabla d (1+\nu { \gamma-\gamma_1 \over \gamma C(x)} d^{\gamma_1}) + DC(x)d^{-\gamma}$, 
$$  \begin{array}{rl}
  & D^2 \varphi (x)=   \gamma\,  ( \gamma+1) C(x) d^{-\gamma-2} \left( 1+ \nu \frac{( \gamma-\gamma_1) ( \gamma-\gamma_1+1)}{ \gamma \,( \gamma+1) C(x)} d^{ \gamma_1}\right ) \nabla d \otimes \nabla d\\
  & -\gamma\,  C(x)  d^{-\gamma-1} \left( 1+ \nu \frac{(\gamma-\gamma_1)}{ \gamma\, C(x)} d^{ \gamma_1} \right) D^2 d 
  -\gamma \, d^{-\gamma-1} \left( \nabla d \otimes \nabla C+ \nabla C \otimes \nabla d \right) + d^{-\gamma}D^2C\, .
   \end{array}
   $$
 In particular 
 
 $$
 |D^2\varphi- C(x) \gamma ( \gamma+1)  \ d^{-\gamma-2} (\nabla d \otimes \nabla d)(1+\nu {( \gamma-\gamma_1 )(\gamma+1-\gamma_1 )\over \gamma( \gamma+1)C(x)}d^{\gamma_1}) | \leq c d^{-\gamma-1}
$$
 which implies 
  \begin{eqnarray}\label{majFD2}
  \left\vert F( D\varphi, D^2 \varphi)\right.&-&\left. C(x) \gamma ( \gamma+1)  \ d^{-\gamma-2} (1+\nu {( \gamma-\gamma_1 )(\gamma+1-\gamma_1 )\over \gamma( \gamma+1)C(x)}d^{\gamma_1}) F( D \varphi, \nabla d \otimes \nabla d)\right\vert\nonumber
  \\ &\leq &c d^{-(\gamma+1)(1+\alpha) }= o( d^{-(\gamma+1) \alpha -\gamma-2 + \gamma_1}).
  \end{eqnarray}
  Let $\tau \in ]{ \gamma_1 \over \alpha } ,1[$,  by Property $(P_2)$ ,  one has  ( in the computations below $c$ denotes always some universal constant which varies from one line to another) :
  \begin{eqnarray}\label{majphi}
 \left\vert  F( D \varphi,  \nabla d \otimes \nabla d) \right.&-&
   \left. F( -\gamma  C(x) d^{-\gamma-1} \nabla d (1+\nu { \gamma-\gamma_1 \over \gamma C(x)} d^{\gamma_1}), \nabla d \otimes \nabla d) \right\vert \nonumber\\
    &\leq & c | \nabla d |^2 d^{-(\gamma +1)\alpha} \sum_{i=1}^{i= N}  \left\vert   \left\vert \gamma  C(x)  \partial_i d (1+\nu { \gamma-\gamma_1 \over \gamma C(x)} d^{\gamma_1}) \right\vert^\alpha \right.\nonumber\\
    &-&\left.  \left\vert  \gamma  C(x)  \partial_i d (1+\nu { \gamma-\gamma_1 \over \gamma C(x)} d^{\gamma_1})+ \partial_i C d\right\vert ^\alpha  \right\vert\nonumber \\
    \end{eqnarray}
    \begin{eqnarray}    & \leq&   c  d^{-(\gamma +1)\alpha}   d^{ \tau \alpha} 
    +  c\sum_{i, |\partial_i d| > d^\tau} | \partial_i d |^\alpha\left( \left(1+ { \partial_i C \over \gamma C \partial_i d (1+\nu { \gamma-\gamma_1 \over \gamma C(x)} d^{\gamma_1})} d\right)^\alpha - 1\right)\nonumber\\
       &\leq & c  d^{-(\gamma +1)\alpha}  \left( d^{ \tau \alpha} + \sum_{i,  |\partial_i d| > d^\tau} | \partial_i d |^\alpha  {cd \over | \partial_i d |  } \right) \nonumber\\&\leq & 
    c  d^{-(\gamma +1)\alpha} (d^{ \tau\alpha} + d^{1- \tau ( 1-\alpha)^+})
    \leq   c  d^{-(\gamma +1)\alpha} d^{  \tau \alpha} = o( d^{-( \gamma+1) \alpha + \gamma_1}).
    \end{eqnarray}
 Now observe that by using a Taylor expansion at order $2$ and the properties of $F$ 
 \begin{eqnarray}\label{TE}
&& F( -\gamma  C(x) d^{-\gamma-1} \nabla d (1+\nu { \gamma-\gamma_1 \over \gamma C(x)} d^{\gamma_1}), \nabla d \otimes \nabla d)\nonumber\\
& = & 
   (\gamma  C(x) d^{-\gamma-1})^\alpha (1+\nu \alpha { \gamma-\gamma_1 \over \gamma C(x)} d^{\gamma_1}) F( \nabla d, \nabla d \otimes \nabla d)
   + O( d^{-( \gamma+1)\alpha + 2 \gamma_1}).
   \end{eqnarray}
Gathering \eqref{majFD2}  \eqref{majphi}  and \eqref{TE} one obtains 
\begin{eqnarray*}\label{Fphi}
&&\left\vert F( \nabla \varphi, D^2 \varphi)\right.\\
&-&\left.   (C(x)\gamma)^{1+\alpha}  ( \gamma+1)   \ d^{-\gamma-2} (1+\nu \alpha { \gamma-\gamma_1 \over \gamma C(x)} d^{\gamma_1}) (1+\nu {( \gamma-\gamma_1 )(\gamma+1-\gamma_1 )\over \gamma( \gamma+1)C(x)}d^{\gamma_1}) F( \nabla d, \nabla d \otimes \nabla d) \right\vert \nonumber\\
&=& o( d^{-(\gamma+1)(1+\alpha) -1+ \gamma_1}).
\end{eqnarray*}
    On the other hand  one has, using a Taylor expansion at the order $2$ and the mean value's Theorem 
  \begin{equation}\label{beta}
  \left\vert   | \nabla \varphi |^\beta -  \left( d^{-\gamma-1} C(x) \gamma\right)^\beta  (1+\nu { \beta (\gamma-\gamma_1) \over \gamma C(x)} d^{\gamma_1}) | \nabla d |^\beta \right\vert
  \leq  c (d^{-( \gamma+1) \beta+2\gamma_1}+ d^{-(\gamma+1) \beta +1} ).
  \end{equation}
Considering the term in  $d^{ \gamma_1}d^{-\gamma-2-( \gamma+1)\alpha}  $  in $(C(x)\gamma)^{1+\alpha}  ( \gamma+1)^\alpha   \ d^{-\gamma-2} ((1+\nu\alpha  { \gamma-\gamma_1 \over \gamma C(x)} d^{\gamma_1}) ) (1+\nu {( \gamma-\gamma_1 )(\gamma+1-\gamma_1 )\over \gamma( \gamma+1)C(x)}d^{\gamma_1})$   and in the left hand side of  (\ref{beta}), and  using  the definition of $C(x)$, one has 
 \begin{eqnarray*}
  -(C(x) \gamma)^{1+\alpha}( \gamma+1) \left( \nu\alpha { \gamma-\gamma_1 \over \gamma C(x)} \right. &+&\left.  \nu {( \gamma-\gamma_1) ( \gamma+1-\gamma_1 \over \gamma( \gamma+1) C(x)} \right) F( \nabla d \otimes \nabla d )  \\
  &+ &  (C(x)\gamma)^\beta  \beta \nu { \gamma-\gamma_1 \over \gamma C(x)} \\
  &> &    (C(x)\gamma)^{\beta-1} \nu { (\gamma-\gamma_1)  \gamma_1\over ( \gamma+1) } , 
  \end{eqnarray*}
               and then               
                          \begin{eqnarray*}
                          - F(\nabla \varphi,  D^2 \varphi) + | \nabla \varphi |^\beta &\geq&  d^{ -( \gamma+1)\alpha -\gamma-2 + \gamma_1} (C(x)\gamma)^{\beta-1}\nu { (\gamma-\gamma_1)  \gamma_1\over  ( \gamma+1) } + o( d^{ -( \gamma+1)\alpha -\gamma-2 + \gamma_1} ) .
                                                   \end{eqnarray*}
                                                    Taking $\delta$ small enough one gets  by the assumption on $f$  and $\gamma_1$,  
                         $$- F(\nabla \varphi,  D^2 \varphi) + | \nabla \varphi |^\beta > f^+\geq f. $$
                                                  We then consider in $\delta < d \leq 2 \delta $, as in the proof of Theorem  \ref{exilambda}
              $$\varphi_2  = {C(x)(1+ \nu \delta^{ \gamma_1} )\over \delta ^\gamma} e^{{1 \over d-2\delta}+{1\over \delta}} + D$$
             where,  if $K_1$  is so that $| F( \nabla \varphi_2, , D^2 \varphi_2) | + | \nabla \varphi_2 |^\beta \leq K_1$,  we have denoted $D$ some constant so that 
               $D  \geq \left( { |f|_\infty + K_1 \over \lambda }\right)^{1\over 1+ \alpha}$.
              We have obtained that for $\delta$ small and $d < \delta $, the function 
             $$  \overline{ w}, = \left\{ \begin{array}{lc}
              \varphi + D  & {\rm in} \ d < \delta\\
              \varphi_2 & {\rm in} \ \delta < d < 2\delta\\
              D & {\rm in} \ d > 2 \delta 
              \end{array} \right.$$is a convenient super-solution. 
               
        In the same manner  let 
 $$ \underline{ w}_{ \delta}= \left( \left({  F(  \nabla d , \nabla d \otimes \nabla d) ( \gamma+1)\over \gamma^{\beta-\alpha-1} }\right)^{1\over \beta-\alpha-1} -\nu d^{ \gamma_1} \right) (d+ \delta) ^{-\gamma} - D. $$
 Then    \begin{eqnarray*}
                                        - F( \nabla  \underline{ w}_{ \delta},  D^2 \underline{ w}_{ \delta})&+& |\nabla  \underline{ w}_{ \delta}|^\beta  + \lambda   |\underline{ w}_{ \delta}|^\alpha  \underline{ w}_{ \delta}-f\\
                                        &\leq & - (d+\delta)^{ -( \gamma+1)\alpha -\gamma-2 + \gamma_1} (C(x)\gamma)^{\beta-1}\nu { (\gamma-\gamma_1)  \gamma_1\over  ( \gamma+1) }   \left(1+ o(1)\right)\\
                                        &&+ c \lambda ( d+ \delta)^{-\gamma(1+\alpha)}  -f \\
                                        &\leq & 0,  
                                        \end{eqnarray*}
                                         when $d$ is small, and 
                                        arguing  as previously one can choose $D$ in order that $\underline{ w}_{ \delta}$ be a sub-solution also in $d> \delta$. 
                                        
                                          Now arguing as in the proof of the previous Theorem,  more precisely taking $u_R$  a solution of  \eqref{uR} one gets the existence of $u$ which blows up on the boundary and now is so that 
                                          $ u(x) \sim C(x) d^{-\gamma}$ near the boundary. The uniqueness can be shown as in \cite{BDL2}.

                                        \end{proof}
                                        
We can now prove the Theorem 

\begin{theorem}\label{exiergo}
Under the assumptions of Theorem \ref{exilambda},  and assuming  in addition that $\alpha \geq 2$ and $f$ is Lipschitz continuous, there exists an ergodic  pair  $(u, c)$ , furthermore $u(x) \sim C(x) d^{-\gamma}$ near the boundary.
\end{theorem} 
 \begin{proof} of Theorem \ref{exiergo} :
 By Theorem \ref{exilambda}, for $\lambda>0$ there exists a solution $U_\lambda$ of problem \eqref{eqlambdainfini}, which satisfies estimates \eqref{bou}.  It then follows that $\lambda |U_\lambda|^\alpha U_\lambda$ is locally  bounded in $\Omega$, uniformly with respect to $0<\lambda<1$. Let us fix an arbitrary point $x_0\in \Omega$. Then,   there exists $c\in \mathbb R$ such that, up to a sequence $\lambda_n\to 0$, 
$$\lambda |U_\lambda (x_0)|^\alpha U_\lambda(x_0)\to -c\, .$$ 
On the other hand, Proposition \ref{lipunif} yields that $(U_\lambda)$ is locally uniformly Lipschitz continuous. Therefore, for $x$   in a compact subset of  $\Omega$, one has
 using again  \eqref{bou} and the mean value's Theorem,
$$
\lambda \left| |U_\lambda(x)|^\alpha U_\lambda(x)- |U_\lambda(x_0)|^\alpha U_\lambda(x_0)\right|\leq \lambda \frac{K}{\lambda^\frac{\alpha}{\alpha+1}}|U_\lambda (x)- U_\lambda(x_0)| \to 0 .
$$
It then follows that $c$ does not depend on the choice of $x_0$ and, up to a sequence and locally uniformly in $\Omega$, one has
$$
\lambda |U_\lambda|^\alpha U_\lambda\to -c\, .
$$
Moreover, the function $V_\lambda(x)= U_\lambda (x)-U_\lambda(x_0)$ is locally uniformly bounded,  locally uniformly Lipschitz continuous  and satisfies
$$
- F(\nabla V_\lambda, D^2 V_\lambda)+|\nabla V_\lambda|^\beta=f-\lambda |U_\lambda|^\alpha U_\lambda\ \hbox{ in } \Omega\, .
$$
If $V$ denotes the local uniform limit of $V_\lambda$ for a sequence $\lambda_n\to 0$, then one has
$$
- F(\nabla V, D^2 V)+|\nabla V|^\beta=f+c\ \hbox{ in } \Omega\, .
$$
Let us define for arbitrary $s>0$ : 
$$
\begin{array}{l}
\displaystyle \phi(x)=\frac{\sigma}{(d(x)+s)^\gamma}- \frac{\sigma}{(\delta_o+s)^\gamma} \quad \hbox{ if } \gamma>0\, ,\\[2ex]
\phi(x)=-\sigma\, \log (d(x)+s) + \sigma\, \log (\delta_o+s) \quad \hbox{ if } \gamma=0\, ,
\end{array}
$$and   $\sigma=  \left((\gamma+1)\frac{a}{2} \right)^{\frac{1}{\beta-\alpha-1}}\gamma^{-1}$ if $\gamma>0$,  
$\sigma= \frac{a}{2}$ if $\gamma=0$. Using the computations in Theorem  \ref{exilambda} and Theorem \ref{thcompar}, we have that, for some $\delta_0>0$ sufficiently small,
$$
V_\lambda \geq \phi +\min_{d(x)=\delta_0}V_\lambda\quad \hbox{ in }\Omega\setminus \Omega_{\delta_0}\, ,
$$
 Letting $\lambda, s\to 0$   we deduce that $V(x)\to +\infty$ as $d(x)\to 0$. This shows that $(c,V)$ is an ergodic pair and concludes the proof.
The asymptotic behaviour  can be proved as in \cite{BDL2}. 
  \end{proof}

\section{Proof of Theorem \ref{sympaetpas}.}

\begin{proof}[Proof of Theorem \ref{sympaetpas}] 
Let $u_\lambda$ be a solution of \eqref{eq45}. We begin by giving a bound that will be useful in the whole proof.
Observe that $u_\lambda^+$ is a sub solution of
$$- F( \nabla u_\lambda^+ ,  D^2 u_\lambda^+) \leq |f|_\infty.$$
 By the existence's Theorem in \cite{BD2} let $V$ be a  solution of 
 $$\left\{ \begin{array}{lc} 
 -F( \nabla V, D^2 V) = 2& {\rm in} \ \Omega \\
  V = 0 & {\rm on} \ \partial \Omega\end{array}\right.$$
  Then $V$ is bounded, 
  $|f|^{1\over 1+\alpha } V$ is then a super-solution and  the comparison Theorem  in \cite{BD2} implies that \begin{equation}\label{upper}|u_\lambda ^+|_\infty \leq |V|_\infty  |f|_\infty ^{\frac{1}{1+ \alpha}}\leq c |f|_\infty ^{\frac{1}{1+ \alpha}}.\end{equation}
Let us consider first the case  where  there exists a sub-solution  $\varphi$  for  \eqref{eq3}. 
Then,  $\varphi-| \varphi |_\infty$ is  a sub-solution of  equation \eqref{eq45}, and by the 
comparison principle   we deduce
 $u_\lambda \geq \varphi-| \varphi|_\infty$. Thus, 
in this case $(u_\lambda)$ is uniformly bounded in $\Omega$. The Lipschitz estimates in Theorem \ref{theolip}  then yield that $u_\lambda$ is uniformly converging up to a sequence to a Lipschitz solution of  problem (\ref{eq3}). Note that in this case we did not use $\alpha \geq 2$. 

\medskip
We now treat the second case, i.e. we suppose that  \eqref{eq3} has no  sub-solutions.
In particular $|u_\lambda|_\infty $ diverges, since otherwise by the first part of the proof,  we could extract from $(u_\lambda)$ a subsequence 
 converging to a solution of \eqref{eq3}.  

On the other hand,  since  $-\left(\frac{ |f|_\infty }{\lambda}\right) ^{\frac{1}{1+ \alpha}}$ is a sub solution of \eqref{eq45}, by the comparison principle we obtain 
$u_\lambda^- \leq \left(\frac{|f|_\infty}{\lambda}\right) ^{\frac{1}{1+ \alpha}}$, which,    
jointly with  \eqref{upper}, yields $\lambda |u_\lambda|_\infty^{1+ \alpha} \leq  c_1|f|_\infty.$ 
Hence,   there exists  $(x_\lambda)\subset \Omega$ such that $u_\lambda (x_\lambda)=-|u_\lambda|_\infty\rightarrow -\infty$ and there exists a constant $c_\Omega\geq 0$ such that, up to a subsequence, $\lambda |u_\lambda|_\infty^{1+ \alpha}\rightarrow c_\Omega$.

The rest of the proof follows the lines in \cite{BDL2}.

     \end{proof}

    \begin{proof} of Theorem \ref{ABCDE}
     We do not give the proof which follows the lines in \cite{BDL2}.
     \end{proof}

    \section{A comparison principle  for degenerate  non linear elliptic equations without zero order terms}     \label{compzero} :               
  \begin{theorem}\label{2.4}
Suppose that $b$ is continuous and bounded on $\Omega$  and that either $\alpha = 0$ or $\alpha \neq 0$ and  $f$ is a continuous function such that $f \leq -m<0$.   Suppose that $u$  and $v$ are respectively  a sub-solution  and a super-solution of 
    $$-F( \nabla u,  D^2 u) +b(x) | \nabla u |^\beta = f.$$
Suppose that $u$  or  $v$  is  Lipschitz  and  both the two are bounded  on  $\overline{ \Omega}$,  that $u \leq v$ on $\partial\Omega$. Then $u\leq v$ in $\Omega$. 
\end{theorem}

\begin{proof}Without loss of generality,
we will suppose that $u$ is Lipschitz continuous. 

The case $\alpha = 0$ is quite standard, and is done in \cite{BDL2}, for the sake of shortness we do not reproduce it here.
\medskip

For the case $\alpha \neq 0$ and $f<0$,
we use the change of function $u = \varphi (z)$, $v = \varphi (w)$ 
with $$ \varphi(s) = -\gamma_1 ( \alpha+1) \log \left( \delta + e^{-\frac{s}{\alpha+1}}\right).$$
This function is  used in \cite{BB}, \cite{BM}, \cite{BP},  \cite{LP}, \cite{LPR}.

We choose $\delta$  
small enough in order that the range of $\varphi$  covers the ranges of $u$ and $v$.
The constant $\gamma_1$ will be chosen small enough depending only on $a$, $\alpha$, $\beta$, $\inf_\Omega (-f)$ and $|b|_\infty$; in this proof, any constant of this type will be called universal .  Observe 
that $\varphi^\prime>0$ while $\varphi^{\prime \prime} <0$. 
      Let $Z = \sum_i | \partial_i z |^{\alpha\over 2}  \partial_i z $, $W =  \sum_i | \partial_i w |^{\alpha\over 2}  \partial_i w $. 
In the viscosity sense,   $z$ and $w$  are respectively sub- and super- solution of 
         \begin{equation}\label{eq1}- F( \varphi^\prime (z)^{1+\alpha}\Theta_\alpha (\nabla z)  D^2 z\Theta_\alpha (\nabla z)  + \varphi^\prime(z)^{ \alpha} \varphi^{\prime \prime} (z) ( Z\otimes Z) )   +b(x) \varphi^\prime (z)^{ \beta} | \nabla z |^\beta -f \leq  0.
        \end{equation} 
      
                  \begin{equation}\label{eq1}- F(\varphi^\prime (w)^{1+\alpha}\Theta_\alpha (\nabla w)  D^2 w\Theta_\alpha (\nabla w) +\varphi^\prime(w)^{ \alpha}\varphi^{\prime \prime} (w)( W\otimes W) )    +b(x) \varphi^\prime (w)^{ \beta} | \nabla w |^\beta -f \geq  0.
        \end{equation} 
         We define 
  $$H(x, s, p) =  \frac{-a\varphi^{\prime \prime}(s)}{\varphi^\prime (s)} \sum_i |p_i|^{2+ \alpha} + b(x) \varphi^\prime (s)^{ \beta-\alpha-1} | p |^\beta + \frac{-f(x)}{\varphi^\prime(s) ^{ \alpha+1}}. $$
The point is to prove that  at $\bar x$, a maximum point  of $ z-w$, $\frac{ \partial H ( \bar x, s,p)}{\partial s} >0$ for all 
$p$. This will be sufficient to get a contradiction. 
A simple computation gives
           $$\varphi^\prime =\frac{ \gamma_1  e^{-\frac{s}{\alpha+1}}}{\delta + e^{-\frac{s}{\alpha+1}}},\ 
 \varphi^{\prime \prime} =\frac{ -\gamma_1\delta e^{-\frac{s}{\alpha+1}}}{(\alpha+1) ( \delta + e^{-\frac{s}{\alpha+1}})^2}.$$
Hence 
              $$\left(  \frac{ -\varphi^{\prime \prime}}{ \varphi^\prime } \right) ^\prime = \frac{\delta}{(\alpha +1)^2} 
              \frac{e^{-\frac{s}{\alpha+1}}}{(\delta + e^{-\frac{s}{\alpha+1}})^2}
   \ \mbox{i.e.}\  \left(  \frac{ -\varphi^{\prime \prime }}{\varphi^\prime } \right) ^\prime =  -\frac{\varphi^{ \prime \prime }}{( \alpha +1) \gamma_1}>0.$$
Differentiating $H$ with respect to $s$ gives: 
        $$ \partial _s H = a\sum_i |p_i|^{ \alpha+2}  \frac{ -\varphi^{\prime \prime }}{( \alpha + 1) \gamma_1} 
        + (-f)\frac{ -(\alpha+1) \varphi^{\prime \prime}}{( \varphi^\prime)^{ \alpha+2}} +b(x)  |p|^\beta ( \beta-\alpha-1)( \varphi^\prime )^{ \beta-\alpha-2} \varphi^{\prime \prime}. $$
Since $-\varphi^{ \prime \prime}$  is positive, we need to prove that 
         
$$K:=\frac{ a \sum_i|p_i|^{ \alpha+2}}{( \alpha+1) \gamma_1}+ (-f) {\frac{\alpha+1}{( \varphi^\prime )^{ \alpha+2} }} -|b|_\infty |p|^\beta ( \beta-\alpha-1) (\varphi^\prime )^{ \beta-\alpha-2}>0.$$
We start by treating the case $\beta<\alpha+2$.

Observe first that the boundedness of $u$ and $v$, implies that there exists universal  positive constants $c_o$ and $c_1$ such that 
$$\label{gammab}c_o\gamma_1\leq\varphi^\prime\leq c_1\gamma_1.
$$ 
Hence, it is easy to see that there exist three positive  universal constants $C_1^\prime,$ $C_i$, $i=2,3$  such that
     $$K>\frac{ C_1^\prime \sum_i |p_i|^{ \alpha+2}}{\gamma_1}+\frac{C_2}{\gamma_1^{\alpha+2}}-\frac{C_3(\sum_i |p_i|^2)^{\beta\over 2} }{\gamma_1^{\alpha+2-\beta}}.$$
     
      We now observe that since $\alpha >0$
       $C_1^\prime  \sum_i |p_i|^{\alpha+2} \geq C_1^\prime N^{-\alpha \over 2} (\sum_i p_i^2)^{\alpha+2 \over 2}:=C_1 (\sum_i p_i^2)^{\alpha+2 \over 2}= C_1 |p|^{ \alpha+2} $. 
We choose $\gamma_1=\min \left\{ 1, (\frac{C_3}{C_2})^\beta, (\frac{C_3}{C_1})^\frac{1}{\alpha+1-\beta}\right\}$.    With this choice of $\gamma_1$, 
for $|p|\leq 1$,
$$
\frac{ C_1 |p|^{ \alpha+2}}{\gamma_1}+\frac{C_2}{\gamma_1^{\alpha+2}}-\frac{C_3|p|^\beta}{\gamma_1^{\alpha+2-\beta}} \geq  \frac{C_2}{\gamma_1^{\alpha+2}}-\frac{C_3}{\gamma_1^{\alpha+2-\beta}} >0;
$$
while for $|p|\geq 1$,
$$\frac{ C_1 |p|^{ \alpha+2}}{\gamma_1}+\frac{C_2}{\gamma_1^{\alpha+2}}-\frac{C_3|p|^\beta}{\gamma_1^{\alpha+2-\beta}} \geq   \frac{(C_1) |p|^{ \alpha+2}}{\gamma_1}-\frac{C_3|p|^\beta}{\gamma_1^{\alpha+2-\beta}} >0.
$$
If $\beta = \alpha+2$, just take $\gamma_1 < \frac{a}{( \alpha+1) | b|_\infty}$. 
               
This gives  that for $\gamma_1$ small enough depending only on $\min (-f)$ , $\alpha$, $|b|_\infty$  and $\beta$ one has, for some universal constant $C$,
\begin{equation}\label{gamma}
\partial_s H( x, s, p)\geq C>0.
\end{equation}
We now conclude the proof of the comparison principle.  Suppose by contradiction that $\sup ( z-w) >0$.

We introduce 
$ \psi_j(x, y) = z(x)-w(y)-\frac{j}{2} |x-y|^2$;     , $$  (p_j, X_j)\in \overline{J}^{2,+} z(x_j),
                  \ (p_j, -Y_j)\in \overline{J}^{2,-} w(y_j),\ \mbox{ with}\ p_j=j(x_j-y_j) $$ 
 and 
$$\left( \begin{array}{cc} X_j & 0 \\0& Y_j\end{array} \right) \leq 3 j \left( \begin{array} {cc}I&-I\\-I&I\end{array}\right). $$
On $(x_j, y_j)$, by a continuity argument,  for $j$ large enough one has 
                $$ z(x_j) > w(y_j) +  \frac{\sup(z-w)}{2}.$$ 
Note for later purposes that since $z$ or $w$ are Lipschitz, $ p_j = j ( x_j-y_j)$ is bounded.
Observe that the monotonicity of $\frac{ \varphi^{ \prime \prime }}{\varphi^\prime}$ implies that 
$$N =p_j \otimes p_j \left( \frac{ \varphi^{ \prime \prime } (z(x_j))}{\varphi^\prime ( z(x_j))}- \frac{ \varphi^{ \prime \prime }(w(y_j))}{\varphi^\prime(w(y_j))}\right)\leq 0.$$ 
Using the fact that $z$ and $w$ are respectively sub and super solutions of the equation  (\ref{eq1}), the estimate  \eqref{gamma} and that $H$ is decreasing in the second variable,  one obtains:  
                       \begin{eqnarray*}
     0& \geq & \frac{ -f(x_j)}{\varphi^\prime (z(x_j))^{ \alpha+1} } 
                        -  F\left(p_j, X_j + \frac{ \varphi^{ \prime \prime } (z(x_j))}{\varphi^\prime ( z(x_j))}p_j \otimes p_j\right) + b(x_j)|p_j |^\beta \varphi^\prime ( z(x_j))^{\beta-\alpha-1} \\
     & \geq &\frac{ -f(x_j)}{\varphi^\prime (z(x_j))^{ \alpha+1} }  -F(p_j, -Y_j+  \frac{\varphi^{ \prime \prime } (w(y_j))}{\varphi^\prime ( w(y_j))} p_j \otimes p_j )  \\
   &&+   a \sum_i|(p_j)_i|^{2+ \alpha} \left( \frac{ \varphi^{ \prime \prime } (w(y_j))}{\varphi^\prime ( w(y_j))} -\frac{ \varphi^{ \prime \prime } (z(x_j))}{\varphi^\prime ( z(x_j))}\right)+  |p_j |^\beta  b(x_j)\varphi^\prime ( z(x_j))^{ \beta-\alpha-1}\\
    & \geq & \frac{f(y_j) -f(x_j)}{ \varphi^\prime ( w(y_j))^{ \alpha+1} }+ (b(x_j)-b(y_j)) |p_j |^\beta  \varphi^\prime ( w(y_j)))^{ \beta-\alpha-1}\\
&&+ H(x_j, z(x_j), p_j)- H(x_j, w(y_j),p_j)\\                        
 &\geq&C(z(x_j)-w(y_j))+  \frac{o(1)}{\gamma_1^{\alpha+1}}.          \end{eqnarray*}
Here we have used the continuity of $f$ and $b$, the boundedness of $p_j$ and  that 
$$\psi(x_j,y_j)\geq \sup (\psi(x_j,x_j), \psi(y_j, y_j)).$$ 
Passing to the limit one gets a contradiction, since $(x_j,y_j)$ converges to $(\bar x,\bar x)$ such that $z(\bar x)>w(\bar z)$.  
      \end{proof}
    Theorem \ref{compzero} enables us to pove,  arguing as in \cite{BDL2},  Theorem \ref{ABCDE}.


\begin{thebibliography}{99}
\bibitem{BB} G. Barles, J. Busca{ \em Existence  and  comparison results for fully  non linear degenerate  elliptic equations without zeroth order terms}, Communications in Partial Differential Equations , 26(11-12), (2001), 2323-2337.

\bibitem{BCI} G. Barles,  E. Chasseigne, C. Imbert, {\it H\"older continuity of solutions of second-order non-linear elliptic integro-differential equations} J. Eur. Math. Soc. 
{\bf 13} (2011),  1-26.
\bibitem{BM} G. Barles,  F. Murat, {\em Uniqueness and the maximum principle for quasilinear elliptic equations with quadratic growth conditions.} Archive for rational mechanics and analysis 133.1 (1995), 77-101.
\bibitem{BP} G.Barles, Porretta {\em Uniqueness  for unbounded solutions to stationary viscous Hamilton-Jacobi equation} Ann. Scuola  Norm. Sup Pisa, Cl Sci (5) , Bol V (2006), 107-136. 
\bibitem{BPT}G. Barles, Porretta, A, Tabet, {\em On the large time behavior of solutions of the Dirichlet problem for subquadratic viscous Hamilton-Jacobi equations}  Journal de math\'ematique pures et appliqu\'ees,   94, (2010), 497-519. 

\bibitem{BBr}  P. Bousquet, L.  Brasco,  {\em $\mathcal{C}^1$  regularity of orthotropic p-harmonic functions in the plane}. Anal. PDE 11 (2018), no. 4, 813-854.
\bibitem{BB2} P. Bousquet, L.  Brasco,{\em  Lipschitz regularity for orthotropic Functionals with non standard growth conditions} , Arxiv 1810. 0337v. 
 \bibitem{BBJ} P. Bousquet, L.  Brasco, V. Julin {\em  Lipschitz regularity for local minimizers of some widely degenerate problems}
 . Ann. Sc. Norm. Super. Pisa Cl. Sci. (5) 16 (2016), no. 4, 1235-1274.
\bibitem{BC} L. Brasco, G. Carlier, {\em On certain anisotropic elliptic equations arising in congested optimal transport: Local gradient bounds}  Adv. Calc. Var. 7 (2014), no. 3, 379-407.
\bibitem{BD1}I. Birindelli, F. Demengel \emph{First eigenvalue and Maximum principle for fully nonlinear singular
operators} ,   Advances in Differential equations, Vol
(2006) \textbf{11} (1), 91-119. 

\bibitem{BD2}    I. Birindelli, F. Demengel,\emph
      {Existence and regularity results for fully nonlinear operators on the model of  the  pseudo  Pucci's operators },  J. Elliptic Parabol. Equ. 2 (2016), no. 1-2, 171-187.
      \bibitem{BDcocv} I. Birindelli, F. Demengel, \emph{$\mathcal{C}^{1, \beta} $ regularity for Dirichlet problems associated to fully nonlinear degenerate elliptic equations}. ESAIM Control Optim. Calc. Var. 20 (2014), no. 4, 1009-1024. 
\bibitem{BDL1}I. Birindelli, F. Demengel, F. Leoni, \emph{Dirichlet problems for fully nonlinear  equations with "subquadratic" Hamiltonians}  Contemporary research in elliptic PDEs and related topics, 107–127, Springer INdAM Ser., 33, Springer, Cham, 2019.

\bibitem{BDL2} I. Birindelli, F. Demengel, F. Leoni, \emph{Ergodic pairs for singular or degenerate fully nonlinear operators}  ESAIM Control Optim. Calc. Var. 25 (2019), Art. 75.
\bibitem{BDL3} I. Birindelli, F. Demengel, F. Leoni, { \em On the $\mathcal{C}^{1, \gamma}$ regularity for Fully non linear  singular or degenrate equations  with  a subquadratic hamiltonian},    NoDEA Nonlinear Differential Equations Appl. 26 (2019), no. 5.

\bibitem{CDLP}I. Capuzzo Dolcetta, F. Leoni, A. Porretta, {\em H\"older's estimates for degenerate elliptic equations with coercive Hamiltonian}, Transactions of the american society, Vol 362, n$^o$9, September 2010, pp. 4511-4536. 

\bibitem{I} {H. Ishii},  \emph{ Viscosity solutions  of fully nonlinear equations},  Sugaku Expositions , Vol 9, number 2, December 1996. 
\bibitem{usr} { M.G. Crandall,H. Ishii, P.L. Lions} {\em User's guide to
viscosity solutions of second order partial differential equations}. Bull.
Amer. Math. Soc. (N.S.) 27 (1992), no. 1, 1--67.
\bibitem{D1}F. Demengel, {\em Lipschitz interior  regularity for the  viscosity  and weak solutions of the Pseudo $p$-Laplacian Equation.}
Advances in Differential equations, Vol. 21, Numbers 3-4, March April 2016. 


 \bibitem{D2} F.  Demengel  : {\em Regularity properties  of Viscosity Solutions for Fully Non linear Equations  on the model of the anisotropic $\vec p$Laplacian}.  
 Asymptotic Analysis, vol. 105, no. 1-2, pp. 27-43, 2017. 
 

 \bibitem{FFM} I. Fonseca, N. Fusco, P. Marcellini,   {\em An existence result for a non convex variational problem via regularity}, ESAIM: Control, Optimisation and Calculus of Variations, Vol. 7, ( 2002),  p 69-95.  
\bibitem{IL}{H. Ishii, P.L. Lions}, {\it  Viscosity solutions of Fully-Nonlinear Second  Order Elliptic Partial Differential Equations}, {J. Differential Equations},  \textbf{83},  (1990), 26--78.
\bibitem{LL} J.M. Lasry, P. L. Lions {\em Nonlinear Elliptic Equations with Singular Boundary Conditions and Stochastic Control with state Constraints, } Math. Ann. 283,  (1989), 583-630. 
 \bibitem{LP}T. Leonori, A. Porretta  {\em Large solutions and gradient bounds for quasilinear elliptic equations}, Comm. in Partial Differential Equations 41, 6 (2016), 952-998.
\bibitem{LPR} T. Leonori, A. Porretta, G. Riey,  {\em Comparison principles for p-Laplace equations with lower order terms}, Annali di Matematica Pura ed Applicata 196(3), (2017), 877-903. 
\bibitem{LR} P. Lindqvist, Ricciotti   {\em Regularity for an anisotropic equation in the plane},  Non Linear Analysis, TMA, Vol.177, part B, 2018,  628-636.        
\bibitem{P} A. Porretta, {\em The ergodic limit for a viscous Hamilton- Jacobi equation with Dirichlet conditions},  Rend. Lincei Mat. Appl. 21 (2010). 

\bibitem{UU} N. Uraltseva, N. Urdaletova, {\em The boundedness of the gradients of generalized solutions of degenerate  quasilinear nonuniformly elliptic equations,} Vest. Leningr. UniV. Math, 16, (1984), 263-270, 2,3.  
\end{thebibliography}
          \end{document}